\documentclass[a4paper]{amsart}
\usepackage{amsmath,amsthm,amssymb}
\usepackage[british]{babel}
\usepackage[pdfborder={0 0 0}]{hyperref}
\usepackage{enumerate}
\usepackage{mathrsfs}
\usepackage{thmtools}
\usepackage{microtype}
\usepackage{graphicx}
\usepackage{xypic}
\usepackage{stmaryrd}
\usepackage{subcaption}
\usepackage{float}
\usepackage{tikz}
\usepackage{booktabs}

\newtheorem{thm}{Theorem}[section]

\newtheorem{lem}[thm]{Lemma}
\newtheorem{prop}[thm]{Proposition}
\theoremstyle{definition}
\newtheorem{defi}[thm]{Definition}
\newtheorem{ex}[thm]{Example}
\newtheorem{rem}[thm]{Remark}
\newtheorem{question}[thm]{Question}

\DeclareFontFamily{T1}{pzc}{} 
\DeclareFontShape{T1}{pzc}{m}{it}{<-> s * [1.15] pzcmi8t}{} 
\DeclareMathAlphabet{\mathpzc}{T1}{pzc}{m}{it} 

\newcommand{\A}{\mathcal{A}}
\newcommand{\B}{\mathcal{B}}
\newcommand{\Th}{\mathrm{Th}}

\newcommand{\M}{\mathpzc{M}}
\newcommand{\C}{\mathcal{C}}
\newcommand{\BB}{\mathscr{B}}

\newcommand{\D}{\mathcal{D}}
\newcommand{\E}{\mathcal{E}}
\newcommand{\hiB}{(\B_i)_{i \geq -1}}
\newcommand{\V}{\mathcal{V}}
\newcommand{\U}{\mathcal{U}}

\newcommand{\hA}{(\A_i)_{i \in \omega}}
\newcommand{\hB}{(\B_i)_{i \in \omega}}
\newcommand{\hC}{(\C_i)_{i \in \omega}}

\newcommand{\Pow}{\mathcal{P}}
\newcommand{\Mw}{{\M_w}}
\newcommand{\bC}{\mathbf{C}}
\renewcommand{\P}{\mathcal{P}}
\newcommand{\gM}{\mathfrak{M}}
\newcommand{\gK}{\mathfrak{K}}

\newcommand{\Mh}{{\M_\omega}}
\newcommand{\Mhw}{{\M_{w\omega}}}

\newcommand{\dom}{\mathrm{dom}}

\renewcommand{\phi}{\varphi}

\newcommand{\conc}{%
  \mathord{
    \mathchoice
    {\raisebox{1ex}{\scalebox{.7}{$\frown$}}}
    {\raisebox{1ex}{\scalebox{.7}{$\frown$}}}
    {\raisebox{.7ex}{\scalebox{.5}{$\frown$}}}
    {\raisebox{.7ex}{\scalebox{.5}{$\frown$}}}
  }
}

\begin{document}

\title{First-order logic in the Medvedev lattice}
\author[R. Kuyper]{Rutger Kuyper}
\address[Rutger Kuyper]{Radboud University Nijmegen\\
Department of Mathematics\\
P.O.\ Box 9010, 6500 GL Nijmegen, the Netherlands.}
\email{mail@rutgerkuyper.com}
\thanks{Research supported by NWO/DIAMANT grant 613.009.011 and by 
John Templeton Foundation grant 15619: `Mind, Mechanism and Mathematics: Turing Centenary Research Project'.}
\date{\today}
\maketitle

\begin{abstract}
Kolmogorov introduced an informal calculus of problems in an attempt to provide a classical semantics for intuitionistic logic. This was later formalised by Medvedev and Muchnik as what has come to be called the Medvedev and Muchnik lattices. However, they only formalised this for propositional logic, while Kolmogorov also discussed the universal quantifier. We extend the work of Medvedev to first-order logic, using the notion of a first-order hyperdoctrine from categorical logic, to a structure which we will call the hyperdoctrine of mass problems. We study the intermediate logic that the hyperdoctrine of mass problems gives us, and we study the theories of subintervals of the hyperdoctrine of mass problems in an attempt to obtain an analogue of Skvortsova's result that there is a factor of the Medvedev lattice characterising intuitionistic propositional logic. Finally, we consider Heyting arithmetic in the hyperdoctrine of mass problems and prove an analogue of Tennenbaum's theorem on computable models of arithmetic.
\end{abstract}

\section{Introduction}

In \cite{kolmogorov-1932}, Kolmogorov introduced an interpretation of intuitionistic logic through the use of \emph{problems} (or \emph{Aufgaben}). In this paper, he argued that proving a formula in intuitionistic logic is very much like solving a problem. The exact definition of a problem is kept informal, but he does define the necessary structure on problems corresponding to the logical connectives. His ideas were later formalised by Medvedev \cite{medvedev-1955} as the Medvedev lattice, and a variation of this was introduced by Muchnik \cite{muchnik-1963}.

However, Medvedev and Muchnik only studied propositional logic, while Kolmogorov also briefly discussed the universal quantifier in his paper:

\begin{quote}
``Im allgemeinen bedeutet, wenn $x$ eine Variable (von beliebiger Art) ist und $a(x)$ eine aufgabe bezeichnet, deren Sinn von dem Werte von $x$ abh\"angt, $(x)a(x)$ die Aufgabe ``eine allgemeine Methode f\"ur die L\"osung von $a(x)$ bei jedem einzelnen Wert von $x$ anzugeben''. Man soll dies so verstehen: Die aufgabe $(x)a(x)$ zu l\"osen, bedeutet, imstande sein, f\"ur jeden gegebenen Einzelwert $x_0$ von $x$ die Aufgabe $a(x_0)$ nach einer endlichen Reihe von im voraus (schon vor der Wahl von $x_0$) bekannten Schritten zu l\"osen.''
\end{quote}

In the English translation \cite{kolmogorov-1932-tr} this reads as follows:

\begin{quote}
``In the general case, if $x$ is a variable (of any kind) and $a(x)$ denotes a problem whose meaning depends on the values of $x$, then $(x)a(x)$ denotes the problem ``find a general method for solving the problem $a(x)$ for each specific value of $x$''. This should be understood as follows: the problem $(x)a(x)$ is solved if the problem $a(x_0)$ can be solved for each given specific value of $x_0$ of the variable $x$ by means of a finite number of steps which are fixed in advance (before $x_0$ is set).''
\end{quote}

It is important to note that, when Kolmogorov says that the steps should be fixed before $x_0$ is set, he probably does not mean that we should have one solution that works for every $x_0$; instead, the solution is allowed to depend on $x_0$, but it should do so uniformly. This belief is supported by one of the informal examples of a problem he gives: ``given one solution of $ax^2 + bx + c = 0$, give the other solution''. Of course there is no procedure to transform one solution to the other one which does not depend on the parameters $a$, $b$ and $c$; however, there is one which does so uniformly. More evidence can be found in Kolmogorov's discussion of the law of the excluded middle, where he says that a solution of the problem $\forall a (a \vee \neg a)$, where $a$ quantifies over all problems, should be ``a general method which for any problem $a$ allows one either to find its solution or to derive a contradiction from the existence of such a solution'' and that ``unless the reader considers himself omniscient, he will perhaps agree that [this formula] cannot be in the list of problems that he has solved''. In other words, a solution of $\forall a (a \vee \neg a)$ should be a solution of $a \vee \neg a$ for every problem $a$ which is allowed to depend on $a$, and it should be uniform because we are not omniscient.

In this paper, we will formalise this idea in the spirit of Medvedev. To do this, we will use the notion of a first-order hyperdoctrine from categorical logic, which naturally extends the notion of Brouwer algebras used to give algebraic semantics for propositional intuitionistic logic, to first-order intuitionistic logic. We will give a short overview of the necessary definitions and properties in section \ref{sec-hyperdoc}. After that, in section \ref{sec-omega-med} we will introduce the \emph{degrees of $\omega$-mass problems}, which combine the idea of Medvedev that `solving' should be interpreted as `computing' with the idea of Kolmogorov that `solving' should be uniform in the variables. Using these degrees of $\omega$-mass problems, we will introduce the hyperdoctrine of mass problems in section \ref{sec-med-hyperdoc}. Next, in section \ref{sec-theory} we study the intermediate logic which this hyperdoctrine of mass problems gives us, and we start looking at subintervals of it to try and obtain analogous results to Skvortsova's \cite{skvortsova-1988} remarkable result that intuitionistic propositional logic can be obtained from a factor of the Medvedev lattice. In section \ref{sec-ha} we show that even in these intervals we cannot get every intuitionistic theory, by showing that there is an analogue of Tennenbaum's theorem \cite{tennenbaum-1959} that every computable model of Peano arithmetic is the standard model. Finally, in section \ref{sec-dec} we prove a partial positive result on which theories can be obtained in subintervals of the hyperdoctrine of mass problems, through a characterisation using Kripke models.

Recently, Basu and Simpson \cite{basu-simpson-2014} have independently studied an interpretation of higher-order intuitionistic logic based on the Muchnik lattice. One of the main differences between our approach and their approach is that our approach follows Kolmogorov's philosophy that the interpretation of the universal quantifier should depend uniformly on the variable. On the other hand, in their approach, depending on the view taken either the interpretation does not depend on the quantified variable at all or does so non-uniformly (as we will discuss below in Remark \ref{rem-basu-simpson}). Of course, an important advantage of their approach is that it is suitable for higher-order logic, while we can only deal with first-order logic. Another important difference between our work and theirs is that we start from the Medvedev lattice, while they take the Muchnik lattice as their starting point.

Our notation is mostly standard. We let $\omega$ denote the natural numbers and $\omega^\omega$ the Baire space of functions from $\omega$ to $\omega$. We denote concatenation of strings $\sigma$ and $\tau$ by $\sigma \conc \tau$. For functions $f,g \in \omega^\omega$ we denote by $f \oplus g$ the join of the functions $f$ and $g$, i.e.\ $(f \oplus g)(2n) = f(n)$ and $(f \oplus g)(2n+1) = g(n)$. We let $\langle a_1,\dots,a_n\rangle$ denote a fixed computable bijection between $\omega^n$ and $\omega$. For any set $\A \subseteq \omega^\omega$ we denote by $\overline{\A}$ its complement in $\omega^\omega$. When we say that a set is countable, we include the possibility that it is finite. We denote the join operation in lattices by $\oplus$ and the meet operation in lattices by $\otimes$. A Brouwer algebra is a bounded distributive lattice together with an implication operation $\to$ such that $x \oplus y \geq z$ if and only if $y \geq x \to z$. For unexplained notions from computability theory, we refer to Odifreddi \cite{odifreddi-1989}, for the Muchnik and Medvedev lattices, we refer to the surveys of Sorbi \cite{sorbi-1996} and Hinman \cite{hinman-2012}, for lattice theory, we refer to Balbes and Dwinger \cite{balbes-dwinger-1975}, and finally for unexplained notions about Kripke semantics we refer to Chagrov and Zakharyaschev \cite{chagrov-zakharyaschev-1997} and Troelstra and van Dalen \cite{troelstra-vandalen-1988}.

\section{Categorical semantics for IQC}\label{sec-hyperdoc}

In this section we will discuss the notion of \emph{first-order hyperdoctrine}, as formulated by Pitts \cite{pitts-1989}, based on the important notion of hyperdoctrine introduced by Lawvere \cite{lawvere-1969}. These first-order hyperdoctrines can be used to give sound and complete categorical semantics for IQC (intuitionistic first-order logic). Our notion of first-order logic in the Medvedev lattice will be based on this, so we will discuss the basic definitions and the basic properties before we proceed with our construction. We use the formulation from Pitts \cite{pitts-2002} (but we use Brouwer algebras instead of Heyting algebras, because the Medvedev lattice is normally presented as a Brouwer algebra).

Let us first give the definition of a first-order hyperdoctrine. After that we will discuss an easy example and discuss how first-order hyperdoctrines interpret first-order intuitionistic logic. We will not discuss all details and the full motivation behind this definition, instead referring the reader to the works by Pitts \cite{pitts-1989, pitts-2002}. However, we will discuss some of the motivation behind this definition in Remark \ref{rem-hyp} below.

\begin{defi}{\rm (\cite[Definition 2.1]{pitts-2002})}\label{def-hyperdoctrine}
Let $\bC$ be a category such that for every object $X \in \bC$ and every $n \in \omega$, the $n$-fold product $X^n$ of $X$ exists. A \emph{first-order hyperdoctrine} $\P$ over $\bC$ is a contravariant functor $\P: \bC^{\rm{op}} \to \mathbf{Poset}$ from $\bC$ into the category $\mathbf{Poset}$ of partially ordered sets and order homomorphisms, satisfying:
\begin{enumerate}[\rm (i)]
\item\label{def-hyp-1} For each object $X \in \bC$, the partially ordered set $\P(X)$ is a Brouwer algebra;
\item\label{def-hyp-2} For each morphism $f: X \to Y$ in $\bC$, the order homomorphism $\P(f): \P(Y) \to \P(X)$ is a homomorphism of Brouwer algebras;
\item\label{def-hyp-3} For each diagonal morphism $\Delta_X: X \to X \times X$ in $\bC$ (i.e.\ a morphism such that $\pi_1 \circ \Delta_X = \pi_2 \circ \Delta_X = 1_X$), the right adjoint to $\P(\Delta_X)$ at the bottom element $0 \in \P(X)$ exists. In other words, there is an element ${=_X} \in \P(X \times X)$ such that for all $A \in \P(X \times X)$ we have
\[\P(\Delta_X)(A) \leq 0 \text{ if and only if } A \leq {=_X}.\]
\item\label{def-hyp-4} For each product projection $\pi: \Gamma \times X \to \Gamma$ in $\bC$, the order homomorphism $\P(\pi): \P(\Gamma) \to \P(\Gamma \times X)$ has both a right adjoint $(\exists x)_\Gamma$
and a left adjoint $(\forall x)_\Gamma$, i.e.:
\begin{align*}
\P(\pi)(B) \leq A \text{ if and only if } B \leq (\exists x)_\Gamma(A)\\
A \leq \P(\pi)(B) \text{ if and only if } (\forall x)_\Gamma(A) \leq B.
\end{align*}
Moreover, these adjoints are natural in $\Gamma$, i.e.\ given $s: \Gamma \to \Gamma'$ in $\bC$ we have\label{condiv}
\begin{figure}[H]
\centering
\begin{subfigure}{.45\textwidth}
\xymatrix{\P(\Gamma' \times X)\ar[r]_{\P(s \times 1_X)}\ar[d]_{(\exists x)_{\Gamma'}}&\P(\Gamma \times X)\ar[d]_{(\exists x)_\Gamma}\\
    \P(\Gamma') \ar[r]_{\P(s)}&\P(\Gamma)}
\end{subfigure}
\hspace{0.1\textwidth}
\begin{subfigure}{.45\textwidth}
\xymatrix{\P(\Gamma' \times X)\ar[r]_{\P(s \times 1_X)}\ar[d]_{(\forall x)_{\Gamma'}}&\P(\Gamma \times X)\ar[d]_{(\forall x)_\Gamma}\\
    \P(\Gamma') \ar[r]_{\P(s)}&\P(\Gamma).}
\end{subfigure}
\end{figure}
\noindent This condition is called the \emph{Beck-Chevalley condition}.
\end{enumerate}

We will also denote $P(f)$ by $f^*$.
\end{defi}

\begin{rem}
We emphasise that the adjoints $(\exists x)_\Gamma$ and $(\forall x)_\Gamma$ only need to be order homomorphisms, and that they do no need to preserve the lattice structure. This should not come as a surprise: after all, the universal quantifier does not distribute over logical disjunction, and neither does the existential quantifier distribute over conjunction.
\end{rem}

\begin{ex}{\rm (\cite[Example 2.2]{pitts-2002})}\label{ex-compl-br}
Let $\BB$ be a complete Brouwer algebra. Then $\BB$ induces a first-order hyperdoctrine $\P$ over the category $\mathbf{Set}$ of sets and functions as follows. We let $\P(X)$ be $\BB^X$, which is again a Brouwer algebra under coordinate-wise operations. Furthermore, for each function $f: X \to Y$ we let $\P(f)$ be the function which sends $(B_y)_{y \in Y}$ to the set given by $A_x = B_{f(x)}$. The equality predicates $=_X$ are given by
\begin{align*}
{=_X}(x,z) = \begin{cases}
0 & \text{if } x=z\\
1 & \text{otherwise.}
\end{cases}
\end{align*}
For the adjoints we use the fact that $\BB$ is complete: given $B \in \P(\Gamma \times X)$ we let
\[((\forall x)_\Gamma(B))_\gamma = \bigoplus_{x \in X} B_{(\gamma,x)}\]
and
\[((\exists x)_\Gamma(B))_\gamma = \bigotimes_{x \in X} B_{(\gamma,x)}.\]
Then $\P$ is directly verified to be a first-order hyperdoctrine.
\end{ex}

\begin{rem}\label{rem-basu-simpson}
A special case of Example \ref{ex-compl-br} is when we take $\BB$ to be the Muchnik lattice. In that case we obtain a fragment of the first-order part of the structure studied by Basu and Simpson \cite{basu-simpson-2014} mentioned in the introduction. Let us consider $\Gamma = \{\emptyset\}$ and $X = \omega$. Thus, if we have a sequence of problems $\B_{(\emptyset,0)},\B_{(\emptyset,1)},\dots$ (which we will write as $\B_0,\B_1,\dots$), we have
\[(\forall x)_\Gamma((\B_i)_{i \in \omega}) = \bigoplus_{i \in \omega} \B_i = \left\{f \in \omega^\omega \mid \forall i \in \omega \exists g \in \B_i (f \geq_T g)\right\},\]
in other words a solution of the problem $\forall x (\B(x))$ computes a solution of every $\B_i$ but does so non-uniformly.

If, as in \cite{basu-simpson-2014}, we take each $\B_i$ to be the canonical representative of its Muchnik degree, i.e.\ we take $\B_i$ to be upwards closed under Turing reducibility, then we have that
\[(\forall x)_\Gamma((\B_i)_{i \in \omega}) = \bigoplus_{i \in \omega} \B_i = \bigcap_{i \in \omega} \B_i,\]
i.e.\ a solution of the problem $\forall x (\B(x))$ is a single solution that solves every $\B_i$. Thus, depending on the view one has on the Muchnik lattice, either the solution is allowed to depend on $x$ but non-uniformly, or it is not allowed to depend on $x$ at all.
\end{rem}

Next, let us discuss how first-order intuitionistic logic can be interpreted in first-order hyperdoctrines. Most of the literature on this subject deals with multi-sorted first-order logic; however, to keep the notation easy and because we do not intend to discuss multi-sorted logic in our particular application, we will give the definition only for single-sorted first-order logic.

\begin{defi}{\rm (Pitts \cite[p.\ B2]{pitts-1989})}
Let $\P$ be a first-order hyperdoctrine over $\bC$ and let $\Sigma$ be a first-order language. Then a \emph{structure} $\gM$ for
$\Sigma$ in $\P$ consists of:
\begin{enumerate}[\rm (i)]
\item an object $M \in \bC$ (the universe),
\item a morphism $\llbracket f \rrbracket_\gM : M^n \to M$ in $\bC$ for every $n$-ary function symbol $f$ in $\Sigma$,
\item an element $\llbracket R \rrbracket_\gM \in \P(M^n)$ for every $n$-ary relation in $\Sigma$.\label{def-mod-case-3}
\end{enumerate}
\end{defi}

Case \eqref{def-mod-case-3} is probably the most interesting part of this definition, since it says that elements of $\P(M^n)$ should be seen as generalised $n$-ary predicates on $M$.

\begin{defi}{\rm (\cite[Table 6.4]{pitts-1989})}
Let $t$ be a first-order term in a language $\Sigma$ and let $\gM$ be a structure in a first-order hyperdoctrine $\P$. Let $\vec{x} = (x_1,\dots,x_n)$ be a context (i.e.\ an ordered list of distinct variables) containing all free variables in $t$.
Then we define the interpretation $\llbracket t(\vec{x}) \rrbracket_\gM \in M^n \to M$ inductively as follows:
\begin{enumerate}[\rm (i)]
\item If $t$ is a variable $x_i$, then $\llbracket t(\vec{x}) \rrbracket_\gM$ is the projection of $M^n$ to the $i$th coordinate.
\item If $t$ is $f(s_1,\dots,s_m)$ for $f$ in $\Sigma$, then $\llbracket t(\vec{x}) \rrbracket_\gM$ is $\llbracket f \rrbracket_\gM \circ (\llbracket s_1(\vec{x}) \rrbracket_\gM,\dots,\llbracket s_m(\vec{x}) \rrbracket_\gM)$.
\end{enumerate}
\end{defi}

Thus, we identify a term with the function mapping a valuation of the variables occurring in the term to the value of the term when evaluated at that valuation.

\begin{defi}{\rm (\cite[Table 8.2]{pitts-1989})}
Let $\phi$ be a first-order formula in a language $\Sigma$ and let $\gM$ be a structure in a first-order hyperdoctrine $\P$. Let $\vec{x} = (x_1,\dots,x_n)$ be a context (i.e.\ an ordered list of distinct variables) containing all free variables in $\phi$.
Then we define the interpretation $\llbracket \phi(\vec{x}) \rrbracket_\gM \in \P(M^n)$ (relative to the context $\vec{x}$) inductively as follows:
\begin{enumerate}[\rm (i)]
\item If $\phi$ is $R(t_1,\dots,t_m)$, then $\llbracket \phi(\vec{x}) \rrbracket_\gM$ is $(\llbracket t_1(\vec{x}) \rrbracket_\gM,\dots,\llbracket t_m(\vec{x}) \rrbracket_\gM)^*(\llbracket R \rrbracket_\gM)$.\label{hyp-int-case-1}
\item If $\phi$ is $t_1 = t_2$, then $\llbracket \phi(\vec{x}) \rrbracket_\gM$ is defined as $(\llbracket t_1(\vec{x}) \rrbracket_\gM,\llbracket t_2(\vec{x}) \rrbracket_\gM)^*(=_M)$.\label{hyp-int-case-2}
\item If $\phi$ is $\top$, then $\llbracket \phi(\vec{x}) \rrbracket_\gM$ is defined as $0 \in \P(M^n)$; i.e.\ the smallest element of $\P(M^n)$.
\item If $\phi$ is $\bot$, then $\llbracket \phi(\vec{x}) \rrbracket_\gM$ is defined as $1 \in \P(M^n)$; i.e.\ the largest element of $\P(M^n)$.
\item If $\phi$ is $\psi \vee \theta$, then $\llbracket \phi(\vec{x}) \rrbracket_\gM$ is defined as $\llbracket \psi(\vec{x}) \rrbracket_\gM \otimes \llbracket \theta(\vec{x}) \rrbracket_\gM$.
\item If $\phi$ is $\psi \wedge \theta$, then $\llbracket \phi(\vec{x}) \rrbracket_\gM$ is defined as $\llbracket \psi(\vec{x}) \rrbracket_\gM \oplus \llbracket \theta(\vec{x}) \rrbracket_\gM$.
\item If $\phi$ is $\psi \to \theta$, then $\llbracket \phi(\vec{x}) \rrbracket_\gM$ is defined as $\llbracket \psi(\vec{x}) \rrbracket_\gM \to \llbracket \theta(\vec{x}) \rrbracket_\gM$.
\item If $\phi$ is $\exists y.\psi$, then $\llbracket \phi(\vec{x}) \rrbracket_\gM$ is defined as $(\exists y)_{M^n}(\llbracket \psi(\vec{x},y) \rrbracket_\gM)$.
\item If $\phi$ is $\forall y.\psi$, then $\llbracket \phi(\vec{x}) \rrbracket_\gM$ is defined as $(\forall y)_{M^n}(\llbracket \psi(\vec{x},y) \rrbracket_\gM)$.
\end{enumerate}
\end{defi}

\begin{defi}{\rm (\cite[Definition 8.4]{pitts-1989})}
Let $\phi$ be a formula in a language $\Sigma$ and a context $\vec{x} = (x_1,\dots,x_n)$, and let $\gM$ be a structure in a first-order hyperdoctrine $\P$. Then we say that $\phi(\vec{x})$ is \emph{satisfied} if $\llbracket \phi(\vec{x}) \rrbracket_\gM = 0$ in $\P(M^n)$. We let the \emph{theory} of $\gM$ be the set of sentences which are satisfied in the empty context, i.e.\ those sentences $\phi$ for which $\phi(\emptyset)$ is satisfied, where $\emptyset$ is the empty sequence. We denote the theory by $\mathrm{\Th}(\gM)$. Given a language $\Sigma$, we let the \emph{theory} of $\P$ be the intersection of the theories of all structures $\gM$ for $\Sigma$ in $\P$, and we denote this theory by $\mathrm{Th}(\P)$.
\end{defi}

\begin{rem}\label{rem-hyp}
Let us make some remarks on the definitions given above.
\begin{itemize}
\item As mentioned above, we identify terms $t(\vec{x})$ with functions $\llbracket t(\vec{x}) \rrbracket_\gM$, and $m$-ary predicates $R(y_1,\dots,y_m)$ are elements of $\P(M^n)$. Since we required our category $\bC$ to contain $n$-fold products, if we have terms $t_1,\dots,t_m$, then $(\llbracket t_1(\vec{x})\rrbracket_\gM,\dots,\llbracket t_m(\vec{x})\rrbracket_\gM): M^n \to M^m$, so $(\llbracket t_1(\vec{x})\rrbracket_\gM,\dots,\llbracket t_m(\vec{x})\rrbracket_\gM)^*: \P(M^m) \to \P(M^n)$. This should be seen as the \emph{substitution of $t_1(\vec{x}),\dots,t_m(\vec{x})$ for $y_1,\dots,y_m$}, which explains case \eqref{hyp-int-case-1} and \eqref{hyp-int-case-2}.
\item Quantifiers are interpreted as adjoints, which is an idea due to Lawvere. For example, for the universal quantifier this says that
\[\llbracket\psi\rrbracket_\gM \geq \llbracket \forall x \phi(x) \rrbracket_\gM \Leftrightarrow \llbracket \psi(x) \rrbracket_\gM \geq \llbracket \phi(x) \rrbracket_\gM,\]
where we assume $x$ does not occur freely in $\psi$. Reading $\geq$ as $\vdash$, the two implications are essentially the introduction and elimination rules for the universal quantifier.
\item The Beck-Chevalley condition is necessary to ensure that substitutions commute with the quantifiers (modulo restrictions on bound variables).
\end{itemize}
\end{rem}

Let us introduce a notational convention: \emph{when the structure is clear from the context, we will omit the subscript $\gM$ in $\llbracket - \rrbracket_\gM$}. Having finished giving the definition of first-order hyperdoctrines, let us just mention that they are sound and complete for intuitionistic first-order logic IQC.

\begin{prop}{\rm (\cite[Proposition 8.8]{pitts-1989})}\label{prop-sound}
Structures in first-order hyperdoctrines are \emph{sound} for IQC, i.e.\ the deductive closure of $\mathrm{\Th}(\gM)$ in IQC is equal to $\mathrm{\Th}(\gM)$.
\end{prop}

\begin{thm}{\rm (Pitts \cite[Corollary 5.31]{pitts-2000})}
The class of first-order hyperdoctrines is \emph{complete} for IQC.
\end{thm}

\section{The degrees of $\omega$-mass problems}\label{sec-omega-med}

In this section, we will introduce an extension of the Medvedev lattice, which we will need to define our first-order hyperdoctrine based on the Medvedev lattice. As mentioned in the introduction, Kolmogorov mentioned in his paper that solving the problem $\forall x \phi(x)$ is the same as solving the problem $\phi(x)$ for all $x$, \emph{uniformly in $x$}. We formalise this in the spirit of Medvedev and Muchnik in the following way.

\begin{defi}
An \emph{$\omega$-mass problem} is an element $\hA \in (\Pow(\omega^\omega))^\omega$. Given two $\omega$-mass problems $\hA,\hB$, we say that $\hA$ \emph{reduces} to $\hB$ (notation: $\hA \leq_\Mh \hB$) if there exists a partial Turing functional $\Phi$ such that for every $n \in \omega$ we have $\Phi(n \conc \B_n) \subseteq \A_n$. If both $\hA \leq_\Mh \hB$ and $\hB \leq_\Mh \hA$ we say that $\hA$ and $\hB$ are \emph{equivalent} (notation: $\hA \equiv_\Mh \hB$). We call the equivalence classes of this equivalence the \emph{degrees of $\omega$-mass problems} and denote the set of the degrees of $\omega$-mass problems by $\Mh$.
\end{defi}

\begin{defi}
Let $\hA, \hB$ be $\omega$-mass problems. We say that $\hA$ \emph{weakly reduces} to $\hB$ (notation: $\hA \leq_\Mhw \hB$) if for every sequence $(g_i)_{i \in \omega}$ with $g_i \in \B_i$ there exists a partial Turing functional $\Phi$ such that for every $n \in \omega$ we have $\Phi(n \conc g_n) \in \A_n$. If both $\hA \leq_\Mhw \hB$ and $\hB \leq_\Mhw \hA$ we say that $\hA$ and $\hB$ are \emph{weakly equivalent} (notation: $\hA \equiv_\Mhw \hB$). We call the equivalence classes of weak equivalence the \emph{weak degrees of $\omega$-mass problems} and denote the set of the weak degrees of $\omega$-mass problems by $\Mhw$.
\end{defi}

The next proposition tells us that $\Mh$ is a Brouwer algebra, like the Medvedev lattice.

\begin{prop}\label{prop-mh-brouwer}
The degrees of $\omega$-mass problems form a Brouwer algebra.
\end{prop}
\begin{proof}
We claim that $\Mh$ is a Brouwer algebra under the component-wise operations on $\M$, i.e.\ the operations induced by:
\begin{align*}
(\hA \oplus \hB)_n &= \{f \oplus g \mid f \in \A_n, g \in \B_n\}\\
(\hA \otimes \hB)_n &= 0 \conc \A_n \cup 1 \conc \B_n\\
(\hA \to \hB)_n &= \{e \conc f \mid \forall g \in \A_n (\Phi_e(g \oplus f) \in \B_n).
\end{align*}
The proof of this is mostly analogous to the proof for the Medvedev lattice, so we will only give the proof for the implication.
Let us first show that $\hA \oplus (\hA \to \hB) \geq_\Mh \hB$. Define a Turing functional $\Phi$ by
\[\Phi(n \conc(g \oplus (e \conc f))) = \Phi_e(g \oplus f).\]
Then $\Phi$ witnesses that $\hA \oplus (\hA \to \hB) \geq_\Mh \hB$.

Conversely, let $\hC$ be such that $\hA \oplus \hC \geq_\Mh \hB$. Let $e \in \omega$ be such that $\Phi_e$ witnesses this fact. Let $\phi$ be a computable function sending $n$ to an index for the functional mapping $h$ to $\Phi_e(n \conc h)$. Let $\Psi$ be the functional sending $n \conc f$ to $\phi(n) \conc f$.
Then $\hC \geq_\Mh \hA \to \hB$ through $\Psi$.
\end{proof}

However, it turns out that this fails for $\Mhw$: it is still a distributive lattice, but it is not a Brouwer algebra.

\begin{prop}
The weak degrees of $\omega$-mass problems form a distributive lattice, but not a Brouwer algebra. In particular, they do not form a complete lattice.
\end{prop}
\begin{proof}
It is easy to see that $\Mhw$ is a distributive lattice under the same operations as $\Mh$. Towards a contradiction, assume $\Mhw$ is a Brouwer algebra, under some implication $\to$. Let $f,g \in \omega^\omega$ be two functions of incomparable Turing degree. Let $\hA$ be given by $\A_i = \{h \mid h \equiv_T f\}$ and let $\hB$ be given by $\B_i = \{f \oplus g\}$. For every $j \in \omega$, let $(\C^j_i)_{i \in \omega}$ be given by $\C^j_i = \{g\}$ for $i=j$, and $\C^j_i = \{f \oplus g\}$ otherwise.

Then, for every $j \in \omega$ we have $\hA \oplus (\C^j_i)_{i \in \omega} \geq_\Mhw \hB$: given a sequence $(h_i)_{i \in \omega}$ with $h_i \in \A_i$, let $e$ be such that $\Phi_e(h_j) = f$. Now let $\Phi(n \conc (s \oplus t))$ be $t$ for $n \not= j$ and $\Phi_e(s) \oplus t$ otherwise. This $\Phi$ is the required witness.

So, since we assumed $\to$ makes $\Mhw$ into a Brouwer algebra, we know that every $(\hA \to \hB) \leq_\Mhw (\C^j_i)_{i \in \omega}$ for every $j \in \omega$. Thus, for every $j \in \omega$ there is some $g_j \leq_T g$ in $(\hA \to \hB)_j$.
For every $j \in \omega$, fix a $\sigma_j \in \omega^{<\omega}$ such that there exists an $n \in \omega$ with $\Phi_j(j \conc (\sigma_j \oplus g_j))(n) {\downarrow} \not= (f \oplus g)(n)$, which exists because $g$, and therefore $g_j \leq_T g$, does not compute $f$. Now let $f_j = \sigma_j \conc f$. Then we have $(f_i)_{i \in \omega} \in \hA$ and $(g_i)_{i \in \omega} \in \hA \to \hB$, but for every $j \in \omega$ we have that $\Phi_j(j \conc (f_j \oplus g_j)) \not\in \B_j$. Thus $\hA \oplus (\hA \to \hB) \not\geq_\Mhw \hB$, a contradiction.
\end{proof}

Finally, let us show that $\Mh$ and $\Mhw$ are extensions of the Medvedev and Muchnik lattices, in the sense that the latter embed into the first. Furthermore, we show that the countable products of $\M$ and $\Mw$ are quotients of $\Mh$ and $\Mhw$.

\begin{prop}
There is a Brouwer algebra embedding of $\M$ into $\Mh$ and a lattice embedding of $\Mw$ into $\Mhw$, both given by
\[\alpha(\A)_n = \A.\]
\end{prop}
\begin{proof}
Direct, using the fact that the diagonal of $\Mh$, i.e.\ $\{\hA \in \Mh \mid \forall n,m (\A_n = \A_m)\}$, is isomorphic to the diagonal of $\M^\omega$, which is directly seen to be isomorphic to $\M$. The same holds for $\Mhw$ and $\Mw$.
\end{proof}

\begin{prop}
There is a Brouwer algebra homomorphism of $\Mh$ onto $\M^\omega$ and a lattice homomorphism of $\Mhw$ onto $\Mw^\omega$.
\end{prop}
\begin{proof}
Follows directly from the fact that all operations on $\Mh$ and $\Mhw$ are component-wise, and the fact that the reducibilities on $\Mh$ and $\Mhw$ are stronger than those on $\M^\omega$ respectively $\M_w^\omega$.
\end{proof}

\section{The hyperdoctrine of mass problems}\label{sec-med-hyperdoc}

In this section, we will introduce our first-order hyperdoctrine based on $\M$ and $\Mh$, which we will call the \emph{hyperdoctrine of mass problems} $\P_\M$. We will take the category $\bC$ to be the category with objects $\{1\},\{1,2\},\dots$ and $\omega$, and with functions the computable functions between them. We will define $\P_\M(\omega)$ to be $\Mh$. Now, let us look at how to define $\P_\M(\alpha) = \alpha^*$ for functions $\alpha: \omega \to \omega$.

\begin{defi}
Let $\alpha: \omega \to \omega$. Then $\alpha^*: \Pow(\omega^\omega)^\omega \to \Pow(\omega^\omega)^\omega$ is the function given by
\[(\alpha^*(\hA))_n = \A_{\alpha(n)}.\]
\end{defi}


\begin{prop}\label{prop-star}
Let $\alpha: \omega \to \omega$ be a computable function. Then $\alpha^*$ induces a well-defined function on $\Mh$ by sending $\hA$ to $\alpha^*(\hA)$, which is in fact a Brouwer algebra homomorphism.
\end{prop}
\begin{proof}
We need to show that if $\hA \leq_\Mh \hB$, then $\alpha^*(\hA) \leq_\Mh \alpha^*(\hB)$. Let $\Phi$ witness that $\A \leq_\Mh \B$. Let $\Psi$ be the partial Turing functional sending $n \conc f$ to $\Phi(\alpha(n) \conc f)$. Then $\Psi$ witnesses that $\alpha^*(\hA) \leq_\Mh \alpha^*(\hB)$. That $\alpha^*$ is a Brouwer algebra homomorphism follows easily from the fact that the operations on $\Mh$ are component-wise.
\end{proof}

Next, we will show that for every computable $\alpha$ we have that $\alpha^*$ has both right and left adjoints, which will certainly suffice to satisfy condition \eqref{condiv} of Definition \ref{def-hyperdoctrine}.

\begin{prop}\label{prop-adjoint}
Let $\alpha: \omega \to \omega$ be a computable function. Then $\alpha^*: \Mh \to \Mh$ has a right adjoint $\exists_\alpha$ and a left adjoint $\forall_\alpha$.
\end{prop}
\begin{proof}
Let us first consider the right adjoint. We define:
\[(\exists_\alpha(\hA))_m = \{n \conc f \mid f \in \A_n \wedge \alpha(n) = m\}.\]
Then $\exists_\alpha$ is a well-defined function on $\Mh$.
Namely, assume $\hA \leq_\Mh \hB$, say through $\Phi$. Let $\Psi$ be the partial functional sending $m \conc n \conc h$ to $n \conc \Phi(n \conc h)$, then $\Psi$ witnesses that $\exists_\alpha(\hA) \leq_\Mh \exists_\alpha(\hB)$.

We claim: $\exists_\alpha$ is a right adjoint for $\alpha^*$, i.e.\ $\alpha^*(\hA) \leq_\Mh \hB$ if and only if $\hA \leq_\Mh \exists_\alpha(\hB)$. First, let us assume that $\alpha^*(\hA) \leq_\Mh \hB$; say through $\Phi$.
Let $\Psi$ be the functional sending $j \conc i \conc h$ to $\Phi(i \conc h)$. We claim: for every $m \in \omega$, $\Psi(m \conc (\exists_\alpha(\hB))_m) \subseteq \A_m$. Indeed, let $n \conc f \in (\exists_\alpha(\hB))_m$. Then $\alpha(n) = m$ and $f \in \B_n$. Thus, per choice of $\Phi$ we know that
\[\Psi(m \conc n \conc f) = \Phi(n \conc f) \in \alpha^*(\hA)_n = \A_{\alpha(n)} = \A_m.\]

Conversely, assume $\hA \leq_\Mh \exists_\alpha(\hB)$; say through $\Psi$. Let $\Phi$ be the functional sending $i \conc h$ to $\Psi(\alpha(i) \conc i \conc h)$. Let $n \in \omega$. We claim:
\[\Phi(n \conc \B_n) \subseteq (\alpha^*(\hA))_n = \A_{\alpha(n)}.\]
Indeed, let $f \in \B_n$. Then $n \conc f \in (\exists_\alpha(\hB))_{\alpha(n)}$. Thus:
\[\Phi(n \conc f) = \Psi(\alpha(n) \conc n \conc f) \in \A_{\alpha(n)}.\]

\bigskip
Next, we consider the left adjoint. We define:
\[(\forall_\alpha(\hA))_m = \left\{\bigoplus_{n \in \omega} f_n \mid \forall n \in \omega((\alpha(n) = m \wedge f_n \in \A_n) \vee (\alpha(n) \not= m \wedge f_n = 0))\right\}.\]
Then $\forall_\alpha$ is a well-defined function on $\Mh$, as can be proven in a similar way as for $\exists_\alpha$. We claim that it is a left adjoint for $\alpha^*$, i.e.\ $\hA \leq_\Mh \alpha^*(\hB)$ if and only if $\forall_\alpha(\hA) \leq_\Mh \hB$.

First, assume $\hA \leq_\Mh \alpha^*(\hB)$, say through $\Phi$. Let $m \in \omega$ and let $g \in \B_m$. Now let
\[f = \bigoplus_{n \in \omega} f_n\]
where $f_n = \Phi(n \conc g)$ if $\alpha(n) = m$, and $f_n = 0$ otherwise. Note that, if $\alpha(n) = m$, then $g \in \B_m = (\alpha^*(\hB))_n$, so $\Phi(n \conc g) \in \A_n$. Thus, $f \in (\forall_\alpha(h_a))_m$. Note that this reduction is uniform in $g$ and $m$, so $\forall_\alpha(\hA) \leq_\Mh \hB$.

Conversely, assume $\forall_\alpha(\hA) \leq_\Mh \hB$, say through $\Psi$. Let $n \in \omega$ and let $g \in \alpha^*(\hB)_n = \B_{\alpha(n)}$. Then $\Psi(\alpha(n) \conc g) \in (\forall_\alpha(\hA))_{\alpha(n)}$. Since clearly $\alpha(n) = \alpha(n)$, it follows that $\Psi(\alpha(n) \conc g)^{[n]} \in \A_n$. Again this reduction is uniform in $n$ and $g$, so $\hA \leq_\Mh \alpha^*(\hB)$.
\end{proof}

\begin{rem}\label{rem-proj}
Note that, if $\alpha: \omega \to \omega$ is is the projection to the first coordinate (i.e.\ the function mapping $\langle n,m \rangle$ to $n$), then
\[\forall_\alpha(\hA) \equiv_\Mh \left(\left\{\bigoplus_{m \in \omega} f_m \mid f_m \in \A_{\langle i,m\rangle}\right\}\right)_{i \in \omega}.\]
We will tacitly identify these two.
Similarly,
\[\exists_\alpha(\hA) \equiv_\Mh \left(\left\{m \conc f_m \mid f_m \in \A_{\langle i,m\rangle}\right\}\right)_{i \in \omega}.\]
\end{rem}

We now generalise this notion to include all the functions in our category $\bC$. We will define $\P_\M(\{1,\dots,n\})$ to be the $n$-fold product $\M^n$.

\begin{defi}\label{def-star-1}
Let $X,Y \in \{\{1\},\{1,2\},\dots\} \cup \{\omega\}$. Let $\alpha: X \to Y$ be computable.
Then $\alpha^*: \P_\M(Y) \to \P_\M(X)$ is the function given by
\[\alpha^*((\A))_i = \A_{\alpha(i)}.\]
\end{defi}

\begin{prop}\label{prop-star-1}
The functions from Definition \ref{def-star-1} are well-defined Brouwer algebra homomorphisms.
\end{prop}
\begin{proof}
As in Proposition \ref{prop-star}.
\end{proof}

\begin{prop}\label{prop-adjoint-1}
Let $X,Y \in \{\{1\},\{1,2\},\dots\} \cup \{\omega\}$ and let $\alpha: X \to Y$ be computable. Then $\alpha^*$ has both left and right adjoints.
\end{prop}
\begin{proof}
As in Proposition \ref{prop-adjoint}.
\end{proof}

Thus, everything we have done above leads us to the following definition.

\begin{defi}\label{def-meddoc}
Let $\bC$ be the category with objects $\{1\},\{1,2\},\dots$ and $\omega$ and functions the computable functions between them. Let $\P_\M$ be the functor sending a finite set $\{1,\dots,n\}$ to $\M^n$, $\omega$ to $\Mh$ and $\alpha$ to $\alpha^*$. We call this the \emph{hyperdoctrine of mass problems}.
\end{defi}

We now verify that the remaining conditions of Definition \ref{def-hyperdoctrine} hold for $\P_\M$.

\begin{thm}\label{thm-meddoc}
The functor $\P_\M$ from Definition \ref{def-meddoc} is a first-order hyperdoctrine.
\end{thm}
\begin{proof}
First note that $\bC$ is closed under all $n$-fold products, because $\omega^n$ is isomorphic to $\omega$ through some fixed computable function $\langle a_1,\dots,a_n\rangle$, and similarly $\{1,\dots,m\}^n$ is isomorphic to $\{1,\dots,mn\}$.

We now verify the conditions from Definition \ref{def-hyperdoctrine}. Condition \eqref{def-hyp-1} follows from Proposition \ref{prop-mh-brouwer}. Condition \eqref{def-hyp-2} follows from Proposition \ref{prop-star-1}. For condition \eqref{def-hyp-3}, use the fact that diagonal morphisms are computable together with Proposition \ref{prop-adjoint-1}.
From the same theorem we know that the projections have left and right adjoints. Thus, we only need to verify that the Beck-Chevalley condition holds for them to verify condition \eqref{def-hyp-4}. Consider the diagram

\centerline{
\xymatrix{\P_\M(\Gamma' \times X)\ar[r]_{(s \times 1_X)^*}\ar[d]_{(\exists x)_{\Gamma'}}&\P_\M(\Gamma \times X)\ar[d]_{(\exists x)_\Gamma}\\
    \P_\M(\Gamma') \ar[r]_{s^*}&\P_\M(\Gamma),}}
    
\noindent we need to show that it commutes.

We have:
\[((\exists x)_\Gamma(({s \times 1_X})^*((\A_i)_{i \in \Gamma' \times X})))_n = \{m \conc \langle n,m \rangle \conc f \mid f \in \A_{\langle s(n),m\rangle}\}\]
and
\[(s^*((\exists x)_{\Gamma'}((\A_i)_{i \in \Gamma' \times X})))_n = \{\langle s(n),m \rangle \conc m \conc f \mid f \in \A_{\langle s(n),m\rangle}\}\]
by Remark \ref{rem-proj}. Then $s^*((\exists x)_{\Gamma'}((\A_i)_{i \in \Gamma \times X})) \leq_\Mh ((\exists x)_\Gamma(({s \times 1_X})^*((\A_i)_{i \in \Gamma \times X})))$ through the functional sending $i \conc k \conc \langle n,m \rangle \conc f$ to $\langle s(n),m \rangle \conc m \conc f$, and the opposite inequality holds through the functional sending $n \conc \langle l,m \rangle \conc k \conc f$ to $m \conc \langle n,m \rangle \conc f$.

Next, consider

\centerline{\xymatrix{\P_\M(\Gamma' \times X)\ar[r]_{(s \times 1_X)^*}\ar[d]_{(\forall x)_{\Gamma'}}&\P_\M(\Gamma \times X)\ar[d]_{(\forall x)_\Gamma}\\
    \P_\M(\Gamma') \ar[r]_{s^*}&\P_\M(\Gamma),}}
    
\noindent we need to show that this also commutes.

Again by Remark \ref{rem-proj} we have:
\begin{align*}
((\forall x)_\Gamma(({s \times 1_X})^*((\A_i)_{i \in \Gamma' \times X})))_n &= \left\{\bigoplus_{m \in \omega} f_m \mid f_m \in \A_{\langle s(n),m\rangle}\right\}\\
&= (s^*((\forall x)_{\Gamma'}((\A_i)_{i \in \Gamma' \times X})))_n,
\end{align*}
as desired.
\end{proof}

For future reference, we state the following lemma which directly follows from the formula for the right adjoint given in the proof of Proposition \ref{prop-adjoint}.

\begin{lem}\label{lem-eq}
For any $X$, the equality $=_X$ in $\P_\M$ is given by:
\begin{align*}
(=_X)_{\langle n,m\rangle} = \begin{cases} \omega^\omega & \text{if } n=m\\
\emptyset & \text{otherwise.}
\end{cases}
\end{align*}
\end{lem}
\begin{proof}
From the formula given for the right adjoint in the proof of Proposition \ref{prop-adjoint}, and the definition of $=_X$ in a first-order hyperdoctrine in Definition \ref{def-hyperdoctrine}.
\end{proof}

Finally, let us give an easy example of a structure in $\P_\M$. More examples will follow in the next sections.

\begin{ex}
Consider the language consisting of a constant $0$, a unary function $S$ and binary functions $+$ and $\cdot$. We define a structure $\gM$ in $\P_\M$. Let the universe $M$ be $\omega$. Take the interpretation to be the standard model, i.e.\ $\llbracket 0 \rrbracket = 0$, $\llbracket S \rrbracket (n) = S(n)$, $\llbracket + \rrbracket (n,m) = n+m$ and $\llbracket \cdot \rrbracket (n,m) = n\cdot m$. Then the sentences which hold in $\gM$ are exactly those which have a computable realiser in Kleene's second realisability model, see Kleene and Vesley \cite[p.\ 96]{kleene-vesley-1965}.

Thus, the hyperdoctrine of mass problems can be seen as an extension of Kleene's second realisability model with computable realisers. There is also a topos which can be seen as an extension of this model, namely the Kleene--Vesley topos, see e.g.\ van Oosten \cite{vanoosten-2008}. However, this topos does not follow Kolmogorov's philosophy that the interpretation of the universal quantifier should be uniform in the variable. On the other hand, a topos can interpret much more than just first-order logic.
\end{ex}

Note that our category $\bC$ only contains countable sets. On one hand this could be seen as a restriction, but on the other hand this should not come as a surprise since we are dealing with computability. That it is not that much of a restriction is illustrated by the rich literature on computable model theory dealing with computable, countable models.

\section{Theory of the hyperdoctrine of mass problems}\label{sec-theory}

Given a first-order language $\Sigma$, we wonder what the theory of $\P_\M$ is. In particular, we want to know: is the theory of $\P_\M$ equal to first-order intuitionistic logic IQC? To this, the answer is `no' in general: it is well-known that the weak law of the excluded middle $\neg \phi \vee \neg\neg \phi$ holds in the Medvedev lattice; therefore $\neg \phi \vee \neg\neg \phi$ holds in $\P_\M$ for sentences. However, for the Medvedev lattice we have the following remarkable result by Skvortsova:

\begin{thm}{\rm (Skvortsova \cite{skvortsova-1988})}
There is an $\A \in \M$ such that the propositional theory of $\M / \A$, the quotient of $\M$ by the principal filter generated by $\A$, is $\mathrm{IPC}$.
\end{thm}

Thus, Skortsova's result tells us that there is a principal factor of the Medvedev lattice which captures exactly intuitionistic propositional logic. There is a natural way to extend principal factors to the hyperdoctrine of mass problems: given $\A$ in $\M$, let $\P_{\M / \A}$ be as in Definition \ref{def-meddoc}, but with $\M$ replaced by $\M / \A$, and $\Mh$ replaced by $\Mh / (\A,\A,\dots)$. It is directly verified that $\P_{\M / \A}$ is also a first-order hyperdoctrine. Thus, there is a first-order analogue to the problem studied by Skortsova in the propositional case: is there an $\A \in \M$ such that the sentences that hold in $\P_{\M / \A}$ are exactly those that are deducible in IQC?

First, note that equality is always decidable (i.e.\ $\forall x,y (x = y \vee \neg x=y)$ holds) by the analogue of Lemma \ref{lem-eq} (with $\omega^\omega$ replaced by $\A$). So, can we get the theory to equal IQC plus decidable equality?
Surprisingly, the answer turns out to be `no' in general, even when we look at intervals instead of just factors. The results of this section and the next section are summarised in Table \ref{table-results} below. We study several types of intervals, and for each approach we state a language and a formula $\phi$ which is not true in IQC plus decidable equality, but which is in the theory of every interval of this type.

\begin{table}[h!]
\begin{tabular}{cccc}
\toprule
Proposition & Type (Definition) & Language & Formula\\
\midrule
 & $\P_\M$ (\ref{def-meddoc}) & Unary $R$ & \eqref{formula1}\\
\ref{prop-cd} & $[\B,\A]_\M$ (\ref{def-simple-interval}) & Nullary $R$, unary $S$ & \eqref{formula2}\\
\ref{prop-cd-2} & $[\B,\A]_\M$ (\ref{def-simple-interval}) & Unary $R$ & \eqref{formula3}\\
\ref{thm-not-iqc} & $[\hiB,\A]_{\P_\M}$ (\ref{def-interval}) & Arithmetic & \eqref{formula4}\\
\bottomrule
\end{tabular}
\caption{Formulas not refutable in intervals.}
\label{table-results}
\end{table}
In this table, we use the following formulas:
\begin{equation}\label{formula1}
\forall x(R(x)) \vee \neg\forall x(R(x)),
\end{equation}
\begin{equation}\label{formula2}
\left(\forall x,y,z (x = y \vee x = z \vee y = z) \wedge \forall z(S(z) \vee R)\right) \to \forall z(S(z)) \vee R,
\end{equation}
\begin{equation}\label{formula3}
(\forall x (S(x) \vee \neg S(x))  \wedge \neg \forall x (\neg S(x))) \to \exists x (\neg\neg S(x)),
\end{equation}
\begin{equation}\label{formula4}
T \to \mathrm{Con}(\mathrm{PA}),
\end{equation}
where $T$ is some finite set of formulas derivable in Heyting arithmetic.

Recall that for a poset $X$ and $x,y \in X$ with $x \leq y$ we have that the interval $[x,y]_X$ denotes the set of elements $z \in X$ with $x \leq z \leq y$. If $\BB$ is a Brouwer algebra then so is $[x,y]_\BB$, with lattice operations as in $\BB$ and implication given by
\[u \to_{[x,y]_\BB} v = (u \to_\BB v) \oplus x.\]
If $x=0$, this gives us exactly the factor $\BB / y$.

We can use this to introduce a specific kind of intervals in the hyperdoctrine of mass problems.

\begin{defi}\label{def-simple-interval}
Let $\A,\B \in \M$. Then the \emph{interval} $[\B,\A]_{\P_\M}$ is the first-order hyperdoctrine defined as in Definition \ref{def-meddoc}, but with $\M$ replaced by $[\B,\A]_\M$, and $\Mh$ replaced by $[(\B,\B,\dots),(\A,\A,\dots)]_\Mh$.
\end{defi}

It can be directly verified that this is a first-order hyperdoctrine; if one is not convinced this also follows from the more general Theorem \ref{thm-ival-hyperdoc} below.

The axiom schema CD, consisting of all formulas of the form $\forall z(\phi(z) \vee \psi) \to \forall z(\phi(z)) \vee \psi$, has been studied because it characterises the Kripke frames with constant domain.
Our first counterexample is based on the fact that a specific instance of this schema holds in every structure in an interval of $\M$ with finite universe.

\begin{prop}\label{prop-cd}
Consider the language consisting of one nullary relation $R$, one unary relation $S$ and equality. Then for every interval $[\B,\A]_{\P_\M}$ the formula
\[\left(\forall x,y,z (x = y \vee x = z \vee y = z) \wedge \forall z(S(z) \vee R)\right) \to \forall z(S(z)) \vee R\]
is in $\Th\left([\B,\A]_{\P_\M}\right)$.
However, this formula is not in IQC plus decidable equality.
\end{prop}
\begin{proof}
Let $\gM$ be a structure in $[\B,\A]_{\P_\M}$. Note that by the analogue of Lemma \ref{lem-eq} we know that if $\forall x,y,z (x = y \vee x = z \vee y = z)$ does not hold, then it gets interpreted as $\A$ and then the formula certainly holds. However, $\forall x,y,z (x = y \vee x = z \vee y = z)$ can only hold if $\M$ has at most two elements. Let us first assume $\gM$ has two elements. Let $f$ be an element of $\llbracket \forall z(S(z) \vee R) \rrbracket$. Then $f = f_1 \oplus f_2$, with $f_1 \in \llbracket S(z) \vee R \rrbracket_1$ and $f_2 \in \llbracket S(z) \vee R \rrbracket_2$.\footnote{Note that $\llbracket S(z) \vee R \rrbracket \in \M^2$, so $\llbracket S(z) \vee R \rrbracket_1$ and $\llbracket S(z) \vee R \rrbracket_2$ denote the first and second component.}
There are two cases: either both $f_1$ and $f_2$ start with a $0$ and we can compute an element of $\llbracket\forall z S(z) \rrbracket$, or one of them starts with a $1$ in which case we can compute an element of $\llbracket R \rrbracket$. Since the reduction is uniform in $f$, we see that
\[\llbracket\forall z(S(z) \vee R)\rrbracket \geq_\M \llbracket\forall z(S(z)) \vee R\rrbracket,\]
and thus the formula given in the statement of the proposition holds. If $\gM$ has only one element, a similar proof yields the same result.

To show that the formula is not in IQC, consider the following Kripke frame.
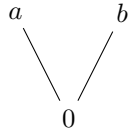
\begin{center}
\begin{tikzpicture}[scale=.7]
  \node (a) at (-1,0) {$a$};
  \node (b) at (1,0) {$b$};
  \node (zero) at (0,-2) {$0$};
  \draw (zero) -- (a);
  \draw (zero) -- (b);
\end{tikzpicture}
\end{center}

Let $\gK_0$ have universe $\{1\}$ and let $\gK_a,\gK_b$ have universe $\{1,2\}$. Let $S(1)$ be true everywhere, let $S(2)$ be true only at $a$ and let $R$ be true only at $b$. Then $\gK$ is a Kripke model refuting the formula in the statement of the proposition.
\end{proof}

Note that the schema CD can be refuted in $\P_\M$, as long as we allow models over infinite structures: namely, let $\phi(z) = S(z)$ and $\psi(z) = R$. We build a structure $\gM$ with $\omega$ as universe. Let $A$ be a computably independent set, i.e.\ for every $n \in \omega$ we have $A \setminus A^{[n]} \not\geq_T A^{[n]}$. Let $\llbracket S \rrbracket_n = A^{[n+1]}$ and let $\llbracket R \rrbracket = A^{[0]}$. Towards a contradiction, assume $\mathrm{CD}$ holds in this structure and let $\Phi$ witness $\llbracket \forall z(S(z) \vee R) \rrbracket \geq_\M \llbracket\forall z(S(z)) \vee R\rrbracket$. Now the function $f$ given by $f^{[n]} = 0 \conc A^{[n+1]}$ is in $\llbracket \forall z(S(z) \vee R) \rrbracket$, so $\Phi(f) \in \llbracket\forall z(S(z)) \vee R \rrbracket$. Because $A$ is computably independent $f$ cannot compute $A^{[0]}$, so $\Phi(f)(0) = 0$. Let $u$ be the use of this computation and let $g$ be the function such that $g^{[n]} = f^{[n]}$ for $n \leq u$ and $g^{[n]} = A^{[0]}$ for $n > u$. Then $\Phi(g)(0) = 0$ so $g$ computes $A^{[u+1]}$, contradicting $A$ being computably independent.

Thus, one might object to our counterexample for being too unnatural by restricting the universe to be finite. However, the next example shows that even without this restriction we can find a counterexample.

\begin{prop}\label{prop-cd-2}
Consider the language consisting of a unary relation $R$. Then for every interval $[\B,\A]_{\P_\M}$ the formula
\[(\forall x (S(x) \vee \neg S(x))  \wedge \neg \forall x (\neg S(x))) \to \exists x (\neg\neg S(x)).\]
is in $\Th\left([\B,\A]_{\P_\M}\right)$.
However, this formula is not in IQC.
\end{prop}
\begin{proof}
Towards a contradiction, assume $\gM$ is some structure satisfying the formula. Let $f \in \llbracket\forall x (S(x) \vee \neg S(x))\rrbracket$ and let $g \in \llbracket \neg \forall x (\neg S(x)) \rrbracket$. If for every $n \in \gM$ we have $f^{[n]}(0) = 1$ then $f$ computes an element of $\llbracket \forall x \neg S(x) \rrbracket$, which together with $g$ computes an element of the top element $\A$ so then we are done. Otherwise we can compute from $f$ some $n \in \gM$ with $f^{[n]}(0) = 0$. Let $\tilde{f}$ be $f^{[n]}$ without the first bit. Let $e$ be an index for the functional sending $(k \conc h_1) \oplus h_2$ to $\Phi_k(h_2 \oplus h_1)$. Then if $k \conc h_1 \in \llbracket \neg S(x) \rrbracket_n$ we have
\[\Phi_e((k \conc h_1) \oplus \tilde{f}) = \Phi_k(\tilde{f} \oplus h_1) \in \A,\]
so $e \conc \tilde{f} \in \llbracket \neg\neg S(x) \rrbracket_n$. Therefore $n \conc e \conc \tilde{f} \in \llbracket \exists x (\neg\neg S(x))\rrbracket$. So 
\[\llbracket\forall x (S(x) \vee \neg S(x))\rrbracket \oplus \llbracket \neg \forall x (\neg S(x)) \rrbracket \geq_\Mh \llbracket \exists x (\neg\neg S(x))\rrbracket.\]

To show that the formula is not in IQC, consider the following Kripke frame.
\begin{center}
\begin{tikzpicture}[scale=.7]
  \node (a) at (0,0) {$a$};
  \node (zero) at (0,-2) {$0$};
  \draw (zero) -- (a);
\end{tikzpicture}
\end{center}

Let $\gK_0$ have universe $\{1\}$ and let $\gK_a$ have universe $\{1,2\}$. Let $S(1)$ be false everywhere and let $S(2)$ be true only at $a$. Then $\gK$ is a Kripke model refuting the formula in the statement of the proposition.
\end{proof}

What the last theorem really says is not that our approach is hopeless, but that instead of looking at intervals $[\B,\A]_{\P_\M}$, we should look at more general intervals. Right now we are taking the bottom element $\B$ to be the same for each $i \in \omega$. Compare this with what happens if in a Kripke model we take the domain at each point to be the same: then $\mathrm{CD}$ holds in the Kripke model. Proposition \ref{prop-cd} should therefore not come as a surprise (although it is surprising that the full schema can be refuted). Instead, we should allow $\B_i$ to vary (subject to some constraints); roughly speaking $\B_i$ then expresses the problem of `showing that $i$ exists' or `constructing $i$'. This motivates the next definition.

\begin{defi}\label{def-interval}
Let $\A \in \M$ and $\hiB \in \Mh$ be such that $(\A,\A,\dots) \geq_\Mh \hB \geq_\Mh (\B_{-1},\B_{-1},\dots)$ and such that $\B_i \not\geq_\M \A$ for all $i \geq -1$. We define the \emph{interval} $[\hiB,\A]_{\P_\M}$ as follows. Let $\bC$ be the category with as objects $\{\{1,\dots,m\}^n \mid n,m \in \omega\} \cup \{\omega,\omega^2,\dots\}$.
\begin{itemize}
\item Let the morphisms in $\bC$ be the computable functions $\alpha$ which additionally satisfy that $\B_{y} \geq_\M \B_{\alpha(y)}$ for all $y \in \dom(\alpha)$ uniformly in $y$, where we define $\B_{(y_1,\dots,y_n)}$ to be $\B_{y_1} \oplus \dots \oplus \B_{y_n}$.
\item We send $\{1,\dots,m\}^n$ to the Brouwer algebra
\[\left[(\B_{a_1} \oplus \dots \oplus \B_{a_n})_{(a_1,\dots,a_n) \in \{1,\dots,m\}^n},(\A,\A,\dots)\right]_{\M^{mn}},\]
and we send $\omega^n$ to the Brouwer algebra 
\[[(\B_{a_1} \oplus \dots \oplus \B_{a_n})_{\langle a_1,\dots,a_n\rangle \in \omega},(\A,\A,\dots,\A)]_\Mh.\]
\item We send every morphism $\alpha: Y \to Z$ to $\P_\M(\alpha) \oplus (\B_i)_{i \in Y}$, i.e.\ the function sending $x$ to $\P_\M(\alpha)(x) \oplus (\B_i)_{i \in Y}$, where we implicitly identify $\omega^n$ with $\omega$ and $\{1,\dots,m\}^n$ with $\{1,\dots,mn\}$ through some fixed computable bijection.
\end{itemize}
\end{defi}

\begin{thm}\label{thm-ival-hyperdoc}
The interval $[\hiB,\A]_{\P_\M}$ is a first-order hyperdoctrine.
\end{thm}
\begin{proof}
First, note that the base category $\bC$ is closed under $n$-fold products: indeed, the $n$-fold product of $Y$ is just $Y^n$, and the projections are computable functions satisfying the extra requirement. Furthermore, if $\alpha_1,\dots,\alpha_n: Y \to Z$ are in $\bC$, then $(\alpha_1,\dots,\alpha_n): Y^n \to Z$ in in $\bC$ because for all $y_1,\dots,y_n \in Y$ we have
\[\B_{(y_1,\dots,y_n)} = \B_{y_1} \oplus \dots \oplus \B_{y_n} \geq_\M \B_{\alpha(y_1)} \dots \oplus \B_{\alpha(y_n)} = \B_{(\alpha_1,\dots,\alpha_n)(y_1,\dots,y_n)},\]
with reductions uniform in $y_1,\dots,y_n$.
Finally, for each $\alpha: Y \to Z$ in $\bC$ we have that $\P_\M(\alpha) \oplus 0_Y$ (where $0_Y$ is the bottom element in the Brouwer algebra to which $Y$ gets mapped) is a Brouwer algebra homomorphism: that joins and meets are preserved follows by distributivity, that the top element is preserved follows directly from $(\A,\A,\dots) \geq_\Mh \hB \geq_\M (\B_{-1},\B_{-1},\dots)$ and that the bottom element is preserved follows from the assumption that $\B_{y} \geq_\M \B_{\alpha(y)}$ for all $y \in \dom(\alpha)$ uniformly in $y$. That implication is preserved is more work: let $\alpha: X \to Y$. Throughout the remainder of the proof we will implicitly identify $\omega^n$ with $\omega$ and $\{1,\dots,m\}^n$ with $\{1,\dots,mn\}$ through some fixed bijection $\langle a_1,\dots,a_n\rangle$. Now:
\begin{align*}
&((\P_\M(\alpha)((\C_i)_{i \in Y}))_j \oplus \B_j) \to_{[\B_j,\A]_\M} ((\P_\M(\alpha)((\D_i)_{i \in Y}))_j \oplus \B_j)\\
=\;\;\;\; &((\C_{\alpha(j)} \oplus \B_j) \to (\D_{\alpha(j)} \oplus \B_j)) \oplus \B_j\\
\equiv_\M\; &(\C_{\alpha(j)} \to \D_{\alpha(j)}) \oplus \B_j\\
=\;\;\;\; &(\P_\M(\alpha)((\C_i)_{i \in Y} \to (\D_i)_{i \in Y}))_j \oplus \B_j,
\end{align*}
with uniform reductions.

Thus, we need to verify that the product projections have adjoints; in fact, we will show that every morphism $\alpha$ in the base category $\bC$ has adjoints. Let $\alpha: X \to Y$. We claim: $\P_\M(\alpha) \oplus (\B_i)_{i \in X}$ has as a right adjoint $\exists_\alpha$ and as a left adjoint the map sending $(\C_i)_{i \in X}$ to $\forall_\alpha((\B_i \to_\M \C_{i})_{i \in X}) \oplus (\B_i)_{i \in Y}$, where $\exists_\alpha$ and $\forall_\alpha$ are as in Proposition \ref{prop-adjoint}. Indeed, we have:
\begin{align*}
(\D_i)_{i \in Y} \leq_\Mh \exists_\alpha((\C_i)_{i \in X}) &\Leftrightarrow
(\D_{\alpha(i)})_{i \in X} \leq_\Mh (\C_i)_{i \in X}
\intertext{and because $(\C_i)_{i \in X} \in [(\B_i)_{i \in X},(\A,\A,\dots)]_\Mh$:}
\Leftrightarrow (\B_i)_{i \in X} \oplus (\D_{\alpha(i)})_{i \in X} \leq_\Mh (\C_i)_{i \in X} &\Leftrightarrow
\P_\M(\alpha)((\D_i)_{i \in Y}) \oplus (\B_i)_{i \in X} \leq_\Mh (\C_i)_{i \in X}.
\end{align*}

Similarly, for $\forall$ we have:
\begin{align*}
&\forall_\alpha((\B_i \to_\M \C_i)_{i \in X}) \oplus (\B_i)_{i \in Y} \leq_\Mh (\D_i)_{i \in Y}\\
\Leftrightarrow &\forall_\alpha((\B_i \to_\M \C_{i})_{i \in X}) \leq_\Mh (\D_i)_{i \in Y}\\
\Leftrightarrow &(\B_i \to_\M \C_{i})_{i \in X} \leq_\Mh (\D_{\alpha(i)})_{i \in X}\\
\Leftrightarrow &(\C_i)_{i \in X} \leq_\Mh (\B_i \oplus \D_{\alpha(i)})_{i \in X}\\
\Leftrightarrow &(\C_i)_{i \in X} \leq_\Mh \P_\M(\alpha)((\D_i)_{i \in Y}) \oplus (\B_i)_{i \in X}.
\end{align*}

Finally, we need to verify that $[\hiB,\A]_{\P_\M}$ satisfies the Beck-Chevalley condition.
We have (writing $\alpha^*$ for the image of the morphism $\alpha$ under the functor for $[\hiB,\A]_{\P_\M}$):
\[((\exists x)_\Gamma(({s \times 1_X})^*((\C_i)_{i \in \Gamma' \times X})))_n = \{m \conc \langle n,m \rangle \conc f \mid f \in \C_{(s(n),m)} \oplus \B_n \oplus \B_m\}\]
and
\[(s^*((\exists x)_{\Gamma'}((\C_i)_{i \in \Gamma' \times X})))_n = \{\langle s(n),m \rangle \conc m \conc f \mid f \in \C_{(s(n),m)} \oplus \B_{n} \oplus \B_{s(n)}\}.\]
As in the proof of Theorem \ref{thm-meddoc} we have
\[s^*((\exists x)_{\Gamma'}((\C_i)_{i \in \Gamma' \times X})) \leq_\Mh ((\exists x)_\Gamma(({s \times 1_X})^*((\C_i)_{i \in \Gamma' \times X}))).\]
The opposite inequality is also almost the same as in the proof of Theorem \ref{thm-meddoc}, except that we now need to use that $\C_{(s(n),m)}$ uniformly computes an element of $\B_{(s(n),m)}$ and hence of $\B_m$.

For the other part of the Beck-Chevalley condition we have:
\begin{align*}
&(((\forall x)_\Gamma((\B_i)_{i \in \Gamma' \times X} \to ({s \times 1_X})^*((\C_{i})_{i \in \Gamma' \times X}))) \oplus (\B_i)_{i \in \Gamma})_n\\
=\;\;\;\; &\left\{\bigoplus_{m \in X} f_m \mid f_m \in \B_n \oplus \B_m \to \left(\B_n \oplus \B_m \oplus \C_{(s(n),m)}\right)\right\} \oplus \B_n\\
\equiv_\M\; &\left\{\bigoplus_{m \in X} f_m \mid f_m \in \B_m \to \C_{(s(n),m)}\right\} \oplus \B_n.\\
\intertext{Now, using the fact that $\B_{s(n)}$ uniformly reduces to $\B_n$:}
\equiv_\M\; &\left\{\bigoplus_{m \in X} f_m \mid f_m \in (\B_{s(n)} \oplus \B_m) \to \C_{(s(n),m)}\right\} \oplus \B_{s(n)} \oplus \B_n\\
=\;\;\;\; &(s^*((\forall x)_{\Gamma'}((\B_i)_{i \in \Gamma' \times X} \to (\C_{i})_{i \in \Gamma' \times X}) \oplus (\B_i)_{i \in \Gamma'}))_n,
\end{align*}
as desired.
\end{proof}

In Propositions \ref{prop-refute-1} and \ref{prop-refute-2} below we will show that we can refute the formulas from Propositions \ref{prop-cd} and \ref{prop-cd-2} in these more general intervals. Next, let us rephrase Lemma \ref{lem-eq} for our intervals.

\begin{lem}\label{lem-eq-ival}
Given any $X$ in the base category $\bC$ of $[\hiB,\A]_{\P_\M}$, let $0_X$ and $1_X$ be the bottom respectively top elements of the Brouwer algebra corresponding to $X$.
Then the equality $=_X$ in $[\hiB,\A]_{\P_\M}$ is given by:
\begin{align*}
(=_X)_{\langle n,m\rangle} = \begin{cases} 0_X & \text{if } n=m\\
1_X & \text{otherwise.}
\end{cases}
\end{align*}
\end{lem}
\begin{proof}
From the formula given for the right adjoint in the proof of Theorem \ref{thm-ival-hyperdoc}, and the definition of $=_X$ in a first-order hyperdoctrine in Definition \ref{def-hyperdoctrine}.
\end{proof}

As a final remark, note that we cannot vary $\A$ (i.e.\ make intervals of the form $[\hiB,(\A_i)_{i \geq -1}]_{\P_\M}$): if we did, then to make $\alpha^*$ into a homomorphism we would need to meet with $\A_i$. While joining with $\B_i$ was not a problem, if we meet with $\A_i$ the implication will in general not be preserved.

\section{Heyting arithmetic in intervals of the hyperdoctrine of mass problems}\label{sec-ha}

In the previous section we introduced the general intervals $[\hiB,\A]_{\P_\M}$. However, it turns out that even these intervals cannot capture every theory in IQC, in the sense that there are deductively closed theories $T$ for which there is no structure in any general interval which has as theory exactly $T$. We will show this by looking at models of Heyting arithmetic. Our approach is based on the following classical result about computable classical models of Peano arithmetic.

\begin{thm}{\rm (Tennenbaum \cite{tennenbaum-1959})}\label{thm-tennenbaum}
There is no computable non-standard model of Peano arithmetic.
\end{thm}
\begin{proof}(Sketch) Let $A,B$ be two c.e.\ sets which are computably inseparable and for which PA proves that they are disjoint (e.g.\ take $A = \{e \in \omega \mid \{e\}(e){\downarrow}=0\}$ and $B=\{e \in \omega \mid \{e\}(e){\downarrow}=1\}$). Let $\phi(e) = \exists s \phi'(e,s)$ define $A$ and let $\psi(e) = \exists s \psi'(e,s)$ define $B$, where $\phi,\psi$ are $\Delta^0_0$-formulas which are monotone in $s$.
Now consider the following formulas:
\begin{align*}
\alpha_1 &= \forall e,s \forall s' \geq s ((\phi'(e,s) \to \phi'(e,s')) \wedge (\psi'(e,s) \to \psi'(e,s')))\\
\alpha_2 &= \forall e,s (\neg (\phi'(e,s) \wedge \psi'(e,s)))\\
\alpha_3 &= \forall n,p \exists! a,b(b < p \wedge ap+b=n)\\
\alpha_4 &= \forall n \exists m \forall e < n (\phi'(e,n) \leftrightarrow \exists a < n(ap_e=m)),
\end{align*}
where $p_e$ denotes the $e$th prime.

These are all provable in PA. The first formula tells us that $\phi'$ and $\psi'$ are monotone in $s$. The second formula expresses that $A$ and $B$ are disjoint. The third formula says that the Euclidean algorithm holds. The last formula tells us that for every $n$, we can code the elements of $A[n] \cap [0,n)$ as a single number. We can prove this inductively, by letting $m$ be the product of those $p_e$ such that $e \in A[n] \cap [0,n)$.

Thus, every non-standard model of Peano arithmetic also satisfies these formulas.
Towards a contradiction, let $\gM$ be a computable non-standard model of PA. Let $n \in M$ be a non-standard element, i.e.\ $n > k$ for every standard $k$. Let $m \in M$ be such that 
\[\gM \models \forall e < n (\phi'(e,n) \leftrightarrow \exists a < n . ap_e=m).\]
If $e \in A$, then $\phi'(e,s)$ holds in the standard model for large enough standard $s$, and since $\gM$ is a model of Robinson's $Q$ and $\phi'$ is $\Delta^0_0$ we see that also $\gM \models \phi'(e,s)$ for large enough standard $s$. By monotonicity, we therefore have $\gM \models \phi'(e,n)$. Thus, $\gM \models \exists a < n . a p_e = m$.

Conversely, if $e \in B$, then $\gM \models \psi'(e,s)$ for large enough standard $s$, so by monotonicity we see that $\gM \models \psi'(e,n)$.
Therefore, $\gM \models \neg\phi'(e,n)$ by $\alpha_2$.
Thus, $\gM \models \neg(\exists a < n . ap_e=m)$.
So, the set $C = \{e \in \omega \mid \gM \models \exists a,b < n . ap_{S^e(0)} = m\}$ separates $A$ and $B$.

However, $C$ is also computable: because the Euclidean algorithm holds in $\gM$, we know that there exist unique $a,b$ with $b < p_{S^e(0)}$ such that $ap_{S^e(0)} + b = m$. Since $\gM$ is computable we can find those $a$ and $b$ computably. Now $e$ is in $C$ if and only if $b = 0$. This contradicts $A$ and $B$ being computably separable.
\end{proof}

When looking at models of arithmetic, we often use that fairly basic systems (like Robinson's $Q$) already represent the computable functions (a fact which we used in the proof of Tennenbaum's theorem above). In other words, this tells us that there is not much leeway to change the truth of $\Delta^0_1$-statements.
The next two lemmas show that in a language without any relations except equality (like arithmetic), as long as our formulas are $\Delta^0_1$, their truth value in the hyperdoctrine of mass problems is essentially classical; in other words, there is also no leeway to make their truth non-classical.

\begin{lem}\label{lem-q-0}
Let $\Sigma$ be a language without relations (except possibly equality).
Let $[\hiB,\A]_{\P_\M}$ be an interval and let $\gM$ be a structure for $\Sigma$ in $[\hiB,\A]_{\P_{\M}}$. Let $\phi(x_1,\dots,x_n)$ be a $\Delta^0_0$-formula and let $a_1,\dots,a_n \in M$. Then we have either $\llbracket \phi(x_1,\dots,x_n) \rrbracket_{\langle a_1,\dots,a_n\rangle} \equiv_\M \B_{-1} \oplus \B_{a_1} \oplus \dots \oplus B_{a_n}$ or $\llbracket \phi(x_1,\dots,x_n) \rrbracket_{\langle a_1,\dots,a_n\rangle} \equiv_\M \A$, with the first holding if and only if
$\phi(a_1,\dots,a_n)$ holds classically in the classical model induced by $\gM$ (i.e.\ the classical model with universe $M$ and functions as in $\gM$).

Furthermore, it is decidable which of the two cases holds, and the reductions between $\llbracket \phi(x_1,\dots,x_n) \rrbracket_{\langle a_1,\dots,a_n\rangle}$ and either $\B_{a_1} \oplus \dots \oplus B_{a_n}$ or $\A$ are uniform in $a_1,\dots,a_n$.
\end{lem}
\begin{proof}
We prove this by induction on the structure of $\phi$.

\begin{itemize}
\item $\phi$ is of the form $t(x_1,\dots,x_n) = s(x_1,\dots,x_n)$: by Lemma \ref{lem-eq-ival} we know that $\llbracket t(x_1,\dots,x_n) = s(x_1,\dots,x_n) \rrbracket_{\langle a_1,\dots,a_n\rangle}$ is either $\B_{-1} \oplus \B_{a_1} \oplus \dots \oplus B_{a_n}$ or $\A$, with the first holding if and only if $t(a_1,\dots,a_n) = s(a_1,\dots,a_n)$ holds classically. Since all functions are computable and equality is true equality, it is decidable which of the two cases holds.
\item $\phi$ is of the form $\psi(x_1,\dots,x_n) \wedge \chi(x_1,\dots,x_n)$: there are three cases:
\begin{itemize}
\item If both $\llbracket \psi(x_1,\dots,x_n) \rrbracket_{\langle a_1,\dots,a_n\rangle}$ and $\llbracket \chi(x_1,\dots,x_n) \rrbracket_{\langle a_1,\dots,a_n\rangle}$ are equivalent to $\B_{-1} \oplus \B_{a_1} \oplus \dots \oplus B_{a_n}$,
then $\llbracket \phi(x_1,\dots,x_n)\rrbracket_{\langle a_1,\dots,a_n\rangle} \equiv_\M \B_{-1} \oplus \B_{a_1} \oplus \dots \oplus B_{a_n}$,
\item If $\llbracket \psi(x_1,\dots,x_n) \rrbracket_{\langle a_1,\dots,a_n\rangle} \equiv_\M \A$, then $\llbracket \phi(x_1,\dots,x_n)\rrbracket_{\langle a_1,\dots,a_n\rangle} \equiv_\M \A$ by sending $f \oplus g$ to $f$,
\item If $\llbracket \chi(x_1,\dots,x_n) \rrbracket_{\langle a_1,\dots,a_n\rangle} \equiv_\M \A$, then $\llbracket \phi(x_1,\dots,x_n)\rrbracket_{\langle a_1,\dots,a_n\rangle} \equiv_\M \A$ by sending $f \oplus g$ to $g$.
\end{itemize}
This case distinction is decidable because the induction hypothesis tells us that the truth of $\psi$ and $\chi$ is decidable.
\item $\phi$ is of the form $\psi(x_1,\dots,x_n) \to \chi(x_1,\dots,x_n)$: this follows directly from the fact that, in any Brouwer algebra with top element $1$ and bottom element $0$, we have $0 \to 1 = 1$ and $0 \to 0 = 1 \to 1 = 1 \to 0 = 0$.
The case distinction is again decidable by the induction hypothesis.
\end{itemize}
The other cases are similar.
\end{proof}

\begin{lem}
Let $\Sigma, \A, \hiB$ and $\gM$ be as in Lemma \ref{lem-q-0}. Let $T$ be some theory which is satisfied by $\gM$, i.e.\ $\llbracket \psi \rrbracket_\gM = \B_{-1}$ for every $\psi \in T$. Let $\phi(x_1,\dots,x_n)$ be a formula which is $\Delta^0_1$ over $T$ and let $a_1,\dots,a_n \in M$. Then either $\llbracket \phi(x_1,\dots,x_n) \rrbracket_{\langle a_1,\dots,a_n\rangle} \equiv_\M \B_{-1} \oplus \B_{a_1} \oplus \dots \B_{a_n}$ or $\llbracket \phi(x_1,\dots,x_n) \rrbracket_{\langle a_1,\dots,a_n\rangle} \equiv_\M \A$, with the first holding if and only if
$\phi(a_1,\dots,a_n)$ holds classically in $\gM$.

Furthermore, it is decidable which of the two cases holds, and the reductions are uniform in $a_1,\dots,a_n$.
\end{lem}
\begin{proof}
Let
\[\phi \Leftrightarrow \forall y_1,\dots,y_m \psi(x_1,\dots,x_n,y_1,\dots,y_m) \Leftrightarrow \exists y_1,\dots,y_m \chi(x_1,\dots,x_n,y_1,\dots,y_m),\]
where $\psi$ and $\chi$ are $\Delta^0_0$-formulas. Then by soundness (see Proposition \ref{prop-sound}) we know that
\begin{equation}\label{eqn-237}
\llbracket \forall y_1,\dots,y_m \psi(x_1,\dots,x_n,y_1,\dots,y_m) \rrbracket \equiv_\Mh \llbracket \exists y_1,\dots,y_m \chi(x_1,\dots,x_n,y_1,\dots,y_m)\rrbracket.
\end{equation}
Let $a_1,\dots,a_n \in M$. We claim: there are some $b_1,\dots,b_m$ such that either
\begin{equation}\label{eqn-238}
\llbracket \psi(x_1,\dots,x_n,y_1,\dots,y_m)\rrbracket_{\langle a_1,\dots,a_n,b_1,\dots,b_m\rangle} \equiv_\M \A
\end{equation}
or
\begin{equation}\label{eqn-239}
\llbracket\chi(x_1,\dots,x_n,y_1,\dots,y_m)\rrbracket_{\langle a_1,\dots,a_n,b_1,\dots,b_m\rangle} \equiv_\M \B_{a_1} \oplus \dots \oplus \B_{a_n} \oplus \B_{b_1} \oplus \dots \oplus \B_{b_m}.
\end{equation}
Indeed, otherwise we see from Lemma \ref{lem-q-0} and some easy calculations that
\[\llbracket \forall y_1,\dots,y_m \psi(x_1,\dots,x_n,y_1,\dots,y_m) \rrbracket_{\langle a_1,\dots,a_n\rangle} \equiv_\M \B_{-1} \oplus \B_{a_1} \oplus \dots \oplus \B_{a_n}\]
and
\[\llbracket \exists y_1,\dots,y_m \chi(x_1,\dots,x_n,y_1,\dots,y_m) \rrbracket_{\langle a_1,\dots,a_n\rangle} \equiv_\M \A,\]
which contradicts \eqref{eqn-237}.

Thus, again by Lemma \ref{lem-q-0}, we can find $b_1,\dots,b_m$ computably such that either \eqref{eqn-238} or \eqref{eqn-239} holds. First, if \eqref{eqn-238} holds, then it can be directly verified that
\[\llbracket \forall y_1,\dots,y_m \psi(x_1,\dots,x_n,y_1,\dots,y_m) \rrbracket_{\langle a_1,\dots,a_n\rangle} \equiv_\M \A,\]
while if \eqref{eqn-239} holds, then it can be directly verified that
\[\llbracket \exists y_1,\dots,y_m \chi(x_1,\dots,x_n,y_1,\dots,y_m) \rrbracket_{\langle a_1,\dots,a_n\rangle} \equiv_\M \B_{-1} \oplus \B_{a_1} \oplus \dots \oplus \B_{a_n},\]
with all the reductions uniform in $a_1,\dots,a_n$.
\end{proof}

Next, we slightly extend this to $\Pi^0_1$-formulas and $\Sigma^0_1$-formulas, although at the cost of dropping the uniformity.
\begin{lem}\label{lem-q-1}
Let $\Sigma, \A, \hiB$ and $\gM$ be as in Lemma \ref{lem-q-0}. Let $\phi(x_1,\dots,x_n)$ be a $\Pi^0_1 $-formula or a $\Sigma^0_1$-formula and let $a_1,\dots,a_n \in M$. Then we have either $\llbracket \phi(x_1,\dots,x_n) \rrbracket_{\langle a_1,\dots,a_n\rangle} \equiv_\M \B_{-1} \oplus \B_{a_1} \oplus \dots \B_{a_n}$ or $\llbracket \phi(x_1,\dots,x_n) \rrbracket_{\langle a_1,\dots,a_n\rangle} \equiv_\M \A$, with the first holding if and only if
$\phi(a_1,\dots,a_n)$ holds classically in $\gM$.\footnote{However, unlike the previous two lemmas, the reductions need not be uniform in $a_1,\dots,a_n$.}
\end{lem}
\begin{proof}
Let $\phi(x_1,\dots,x_n) = \forall y_1,\dots,y_m \psi(x_1,\dots,x_n,y_1,\dots,y_n)$ with $\psi$ a $\Delta^0_0$-formula. First, let us assume $\phi(a_1,\dots,a_n)$ holds classically. Thus, for all $b_1,\dots,b_m \in M$ we know that $\psi(a_1,\dots,a_n,b_1,\dots,b_m)$ holds classically. By Lemma \ref{lem-q-0} we then know that $\psi(a_1,\dots,a_n,b_1,\dots,b_m)$ gets interpreted as $\B_{a_1} \oplus \dots \oplus \B_{a_n} \oplus \B_{b_1} \oplus \dots \oplus \B_{b_m}$ (by a reduction uniform in $b_1,\dots,b_m$). Now note that
\begin{align*}
&\;\;\;\;\llbracket \phi \rrbracket_{\langle a_1,\dots,a_n\rangle}\\
\equiv_\M\;& \bigoplus_{\langle b_1,\dots,b_m \rangle \in \omega} \Big((\B_{a_1} \oplus \dots \oplus \B_{a_n} \oplus \B_{b_1} \oplus \dots \oplus \B_{b_m})\\
&\;\;\to_\M \llbracket \psi(x_1,\dots,x_n,y_1,\dots,y_m)\rrbracket_{\langle a_1,\dots,a_n,b_1,\dots,b_m\rangle}\Big) \oplus (\B_{-1} \oplus \B_{a_1} \oplus \dots \oplus \B_{a_n})\\
\equiv_\M\;& \B_{-1} \oplus \B_{a_1} \oplus \dots \oplus \B_{a_n}.
\end{align*}

Now, let us assume $\phi(a_1,\dots,a_n)$ does not hold classically. Let $b_1,\dots,b_m \in M$ be such that $\psi(a_1,\dots,a_n,b_1,\dots,b_m)$ does not hold classically. By Lemma \ref{lem-q-0} we know that $\psi(a_1,\dots,a_n,b_1\dots,b_m)$ gets interpreted as $\A$. Then it is directly checked that in fact
\[\llbracket \phi(x_1,\dots,x_n)\rrbracket_{\langle a_1,\dots,a_n\rangle} \geq_\M \A,\]
as desired.

The proof for $\Sigma^0_1$-formulas $\phi$ is similar.
\end{proof}

Now, we will prove an analogue of Theorem \ref{thm-tennenbaum} for the hyperdoctrine of mass problems.

\begin{thm}\label{thm-not-iqc}
Let $\Sigma$ be the language of arithmetic consisting of a function symbol for every primitive recursive function, and equality.
There is a finite set of formulas $T \supset Q$ derivable in Heyting arithmetic such that for every interval $[\hiB,\A]_{\P_\M}$ and every classically true $\Pi^0_1$-sentence or $\Sigma^0_1$-sentence $\chi$ we have that every structure $\gM$ in $[\hiB,\A]_{\P_{\M}}$ satisfies $\bigwedge T \to \chi$. In particular this holds for $\chi = \mathrm{Con}(\mathrm{PA})$ and so for this language of arithmetic we have $\mathrm{Th}\left([\hiB,\A]_{\P_{\M}}\right) \not= \mathrm{IQC}$.
\end{thm}
\begin{proof}
Our proof is inspired by the proof of Theorem \ref{thm-tennenbaum} given above. Let $A$, $B$, $\phi'$ and $\psi'$ as in that proof. We first define a theory $T'$ which consists of $Q$ together with the formulas
\begin{align*}
&\forall e,s \forall s' \geq s ((\phi'(e,s) \to \phi'(e,s')) \wedge (\psi'(e,s) \to \psi'(e,s')))\\
&\forall n,s (\neg \phi'(n,s) \wedge \psi'(n,s))\\
&\forall n,p \exists! a,b(b < p \wedge ap+b=n)\\
&\forall n \exists m \forall k,s < n (\phi'(k,s) \leftrightarrow \exists a,b < n . ap_k=m).
\end{align*}
Then $T'$ is deducible in Peano arithmetic; in particular it holds in the standard model. Note that $T'$ is equivalent to a $\Pi^0_2$-formula. Furthermore, note that there are computable Skolem functions (for example, take the function mapping $n$ to the least witness). Thus, we can get rid of the existential quantifiers; for example, we can replace 
\[\forall n,p \exists a,b(b < p \wedge ap+b=n)\]
by
\[\forall n,p (g(n,p) < p \wedge f(n,p)p+g(n,p)=n)\]
where $f$ is the symbol representing the primitive recursive function sending $(n,p)$ to $n$ divided by $p$, and $g$ is the symbol representing the primitive recursive function sending $(n,p)$ to the remainder of the division of $n$ by $p$. We can also turn $Q$ into a $\Pi^0_1$-theory using the predecessor function.

So, let $T$ consist of a $\Pi^0_1$-formula which is equivalent to $T'$, together with $\Pi^0_1$ defining axioms for the finitely many computable functions we used. Then $T$ is certainly deducible in PA, but it is also deducible in Heyting arithmetic because every $\Pi^0_2$-sentence which is in PA in also in HA, see e.g.\ Troelstra and van Dalen \cite[Proposition 3.5]{troelstra-vandalen-1988}.

Now, if $\llbracket \bigwedge T \rrbracket \equiv_\M \A$, we are done. We may therefore assume this is not the case. Then, by Lemma \ref{lem-q-1} we see that $T$ holds classically in $\gM$. Therefore $T'$ also holds classically in $\gM$, and by the proof of Theorem \ref{thm-tennenbaum} we see that $\gM$ is classically the standard model. Therefore $\chi$ holds classically in $\gM$ so we see by Lemma \ref{lem-q-1} that $\llbracket \chi \rrbracket \equiv_\M \B_{-1}$.
\end{proof}

\section{Decidable frames}\label{sec-dec}

In the last section we saw that there are languages such that even for every interval $[\hiB,\A]_{\P_\M}$ we have that $\Th\left([\hiB,\A]_{\P_\M}\right) \not= \mathrm{IQC}$. However, note that Heyting arithmetic, like Peano arithmetic is undecidable. We therefore wonder: what happens if we look at decidable theories? In the classical case, we know that every decidable theory has a decidable model. The intuitionistic case was studied by Gabbay \cite{gabbay-1976} and Ishihara, Khoussainov and Nerode \cite{ikn-1998-2,ikn-1998}, culminating in the following result.

\begin{defi}
A Kripke model is \emph{decidable} if the underlying Kripke frame is computable, the universe at every node is computable and the forcing relation
\[w \Vdash \phi(a_1,\dots,a_n)\]
is computable.
\end{defi}

\begin{defi}
A theory is \emph{decidable} if its deductive closure is computable and equality is decidable, i.e.\
\[\forall x,y (x = y \vee \neg x=y)\]
holds.
\end{defi}

\begin{thm}{\rm (\cite[Theorem 5.1]{ikn-1998-2})}\label{thm-ikn}
Every decidable theory $T$ has a decidable Kripke model, i.e.\ a decidable Kripke model whose theory is exactly the set of sentences deducible from $T$.\footnote{In \cite{ikn-1998-2} this result is stated for first-order languages without equality and function symbols. However,  we can apply the original result to the language with an additional binary predicate $R$ representing equality and to the theory $T'$ consisting of $T$ extended with the equality axioms. Using this equality we can now also represent functions by relations in the usual way.}
\end{thm}

Our next result shows how to encode such decidable Kripke models in intervals of the hyperdoctrine of mass problems. Unfortunately we do not know how to deal with arbitrary decidable Kripke frames; instead we have to restrict to those without infinite ascending chains. As we will see later in this section, this nonetheless still proves to be useful.

\begin{thm}\label{thm-dec-fram}
Let $\gK$ be a decidable Kripke model which is based on a Kripke frame without infinite ascending chains. Then there is an interval $[\hiB,\A]_{\P_\M}$ and a structure $\gM$ in $[\hiB,\A]_{\P_\M}$ such that the theory of $\gM$ is exactly the theory of $\gK$.

Furthermore, if we allow infinite ascending chains, then this still holds for the fragments of the theories without universal quantifiers.
\end{thm}
\begin{proof}
Let $T = \{t_0,t_1,\dots\}$ be a computable representation of the poset $T$ on which $\gK$ is based. Let $f_0,f_1,\dots$ be an antichain in the Turing degrees and let $\D = \{g \mid \exists i (g \leq_T f_i)\}$. Consider the collection $\V = \{C(\{f_i \mid i \in I\} \cup \overline{\D}) \mid I \subseteq \omega\}$. By Kuyper \cite[Theorem 3.3]{kuyper-2014}, this is a sub-implicative semilattice of $[C(\{f_i \mid i \in \omega\}) \cup \overline{\D},\overline{\D}]_\M$. We will use the mass problems $C(\{f_i \mid i \not= j\}) \cup \overline{\D}$ to represent the points $t_j$ of the Kripke frame $T$. If $T$ were finite, we would only have to consider a finite sub-upper semilattice of $\V$, and by Skvortsova \cite[Lemma 2]{skvortsova-1988} the meet-closure of this would be exactly the Brouwer algebra of upwards closed subsets of $T$. However, since in our case $T$ might be infinite, we need to suitably generalise this to arbitrary `meets'.

Let us now describe how to do this. First, we define $\A$:
\begin{align*}
\A = \{k_1 \conc k_2 \conc \left(C(\{f_i \mid i \not\in \{k_1,k_2\}\}) \cup \overline{\D}\right) \mid \text{$t_{k_1}$ and $t_{k_2}$ are incomparable})\}
\end{align*}
if $T$ is not a chain, and $\A = \overline{\D}$ otherwise.
The idea behind $\A$ is that if $t_{k_1}$ and $t_{k_2}$ are incomparable in $T$, then there should be no mass problem representing a point above their representations.

Now, let $\U$ be the collection of upwards closed subsets of $T$. We then define the map $\alpha: \U \to \M$ by:
\[\alpha(Y) = \bigcup\{j \conc \left(\left(C(\{f_i \mid i \not= j\}) \cup \overline{\D}\right) \otimes \A\right) \mid t_j \in Y\},\]
and $\alpha(\emptyset) = \A$.
Now let $\B_{-1} = \alpha(T)$ and let $\B_i = \alpha(Z_i)$, where $Z_i$ is the set of nodes where $i$ is in the domain of $\gK$. Then $\alpha: \U \to [\B_{-1},\A]$ as a function; we are not yet claiming that it preserves the Brouwer algebra structure. We will prove a stronger result for a suitable sub-collection of $\U$ below.

First, let us show that $\alpha$ is injective. Indeed, assume $\alpha(Y) \leq_\M \alpha(Z)$. We will show that $Y \supseteq Z$. By applying Lemma \ref{lem-inf-meet-upw} below twice we then have that for every $j$ with $t_j \in Z$ there exists a $k$ with $t_k \in Y$ such that either $C(\{f_i \mid i \not= k\}) \cup \overline{\D}  \leq_\M C(\{f_i \mid i \not= j\}) \cup \overline{\D}$ or $\A \leq_\M C(\{f_i \mid i \not= j\}) \cup \overline{\D}$. In the first case, towards a contradiction let us assume that $k \not= j$. Then $f_k$ computes an element of $C(\{f_i \mid i \not= j\}) \cup \overline{\D}$ and therefore $f_k \in C(\{f_i \mid i \not= k\}) \cup \overline{\D}$ since the latter is upwards closed. However, this contradicts the fact that the $f_i$ form an antichain in the Turing degrees. Thus, $k = j$ and therefore $t_j \in Y$.

In the latter case, we have that $C(\{f_i \mid i \not\in \{k_1,k_2\}) \cup \overline{\D} \leq_\M C(\{f_i \mid i \not= j\}) \cup \overline{\D}$ for some $k_1,k_2 \in \omega$ for which $t_{k_1}$ and $t_{k_2}$ are incomparable. Without loss of generality, let us assume that $k_1 \not= j$. Then, reasoning as above, we see that $f_{k_1} \in C(\{f_i \mid i \not\in \{k_1,k_2\}) \cup \overline{\D}$, a contradiction.

For ease of notation, let us assume the union of the universes of $\gK$ is $\omega$; the general case follows in the same way. Let $\gM$ be the structure with functions as in $\gK$, and let the interpretation of a relation $\llbracket R(x_1,\dots,x_n)\rrbracket_{\langle a_1,\dots,a_n\rangle}$ be $\alpha(Y)$, where $Y$ is exactly the set of nodes where $R(a_1,\dots,a_n)$ holds in $\gK$.

We show that $\gM$ is as desired. To this end, we claim: for every formula $\phi(x_1,\dots,x_n)$ and every sequence $a_1,\dots,a_n$,
\[\llbracket \phi(x_1,\dots,x_n) \rrbracket_{\langle a_1,\dots,a_n\rangle} \equiv_\M \alpha(Y),\]
where $Y$ is exactly the set of nodes where $a_1,\dots,a_n$ are all in the domain and $\phi(a_1,\dots,a_n)$ holds in the Kripke model $\gK$. Furthermore, we claim that this reduction is uniform in $a_1,\dots,a_n$ and in $\phi$.
We prove this by induction on the structure of $\phi$.
First, if $\phi$ is atomic, this follows directly from the choice of the valuations, from the fact that $\gK$ is decidable and from Lemma \ref{lem-eq-ival}.

Next, let us consider $\phi(x_1,\dots,x_n) = \psi(x_1,\dots,x_n) \vee \chi(x_1,\dots,x_n)$. Let $U$ be the set of nodes where $\psi(a_1,\dots,a_n)$ holds in $\gK$ and similarly let $V$ be the set of nodes where $\chi(a_1,\dots,a_n)$ holds.
By induction hypothesis and by the definition of the interpretation of $\vee$ we have
\begin{align*}
&\llbracket \psi(x_1,\dots,x_n) \vee \chi(x_1,\dots,x_n) \rrbracket_{\langle a_1,\dots,a_n\rangle}\\
\equiv_\M\;& \alpha(U) \otimes \alpha(V)\\
=\;\;\;\;& \bigcup\{j \conc \left(\left(C(\{f_i \mid i \not= j\}) \cup \overline{\D}\right) \otimes \A\right) \mid t_j \in U\}\\
\;\;\;\;\;\otimes\;&\bigcup\{j \conc \left(\left(C(\{f_i \mid i \not= j\}) \cup \overline{\D}\right) \otimes \A\right) \mid t_j \in V\}.
\end{align*}
We need to show that this is equivalent to
\[\alpha(Y) = \bigcup\{j \conc \left(\left(C(\{f_i \mid i \not= j\}) \cup \overline{\D}\right) \otimes \A\right) \mid t_j \in Y\},\]
where $Y$ is the set of nodes where $\phi(a_1,\dots,a_n)$ holds. First, let $j \conc f \in \alpha(Y)$. Then $\phi(a_1,\dots,a_n)$ holds at $t_j$. Thus, by the definition of truth in Kripke frames, we know that at least one of $\psi(a_1,\dots,a_n)$ and $\chi(a_1,\dots,a_n)$ holds in $t_j$, and because our frame is decidable we can compute which of them holds. So, send $j \conc f$ to $0 \conc j \conc f$ if $\psi(a_1,\dots,a_n)$ holds, and to $1 \conc j \conc f$ otherwise. Thus, $\alpha(U) \otimes \alpha(V) \leq_\M \alpha(Y)$. Conversely, if either $\psi(a_1,\dots,a_n)$ or $\chi(a_1,\dots,a_n)$ holds then $\phi(a_1,\dots,a_n)$ holds, so the functional sending $i \conc j \conc f$ to $j \conc f$ witnesses that $\alpha(Y) \leq_\M \alpha(U) \otimes \alpha(V)$.

The proof for conjunction is similar. Next, let us consider implication. So, let $\phi(x_1,\dots,x_n) = \psi(x_1,\dots,x_n) \to \chi(x_1,\dots,x_n)$. Let $U$ be the set of nodes where $\psi(a_1,\dots,a_n)$ holds in $\gK$, let $V$ be the set of nodes where $\chi(a_1,\dots,a_n)$ holds and let $Y$ be the set of nodes where $\phi(a_1,\dots,a_n)$ holds. By induction hypothesis, we know that
\begin{align*}
\llbracket \phi(x_1,\dots,x_n) \rrbracket_{\langle a_1,\dots,a_n\rangle}\\
\equiv_\M \alpha(U) \to_{[\B_{(a_1,\dots,a_n)},\A]} \alpha(V).
\end{align*}
First, note that $\alpha(Y) \geq_\M \alpha(U) \to_{[\B_{(a_1,\dots,a_n)},\A]} \alpha(V)$ is equivalent to $\alpha(Y) \oplus \alpha(U) \geq_\M \alpha(V)$. So, let $k \conc h \in \alpha(Y)$ and $j \conc g \in \alpha(U)$.
Then $t_k \in Y$, $h \in \left(C(\{f_i \mid i \not= k\}) \cup \overline{\D}\right) \otimes \A$, $t_j \in U$ and $g \in \left(C(\{f_i \mid i \not= j\}) \cup \overline{\D}\right) \otimes \A$. We need to uniformly compute from this some $m \in \omega$ with $t_m \in Y$ and an element of $\left(C(\{f_i \mid i \not\in p_m\}) \cup \overline{\D}\right) \otimes \A$. First, if either the first bit of $h$ or $g$ is $1$, then $h$ or $g$, respectively, computes an element of $\A$. So, we may assume this is not the case. Then there are $i_1 \not= j$ and $i_2 \not= k$ such that $g \geq_T f_{i_1}$ and $h \geq_T f_{i_2}$. If $i_1 \not= i_2$ then $h \oplus g \in \overline{\D}$, and if $i_1 = i_2$ then $h \oplus g \in C(\{f_i \mid i \not\in \{k,j\}\})$. So, we have
\[h \oplus g \in C(\{f_i \mid i \not\in \{k,j\}\}) \cup \overline{\D}.\]
There are now two cases: if $t_k$ and $t_j$ are incomparable then 
$k \conc j \conc (h \oplus g) \in \A$.
Otherwise, compute $m \in \{k,j\}$ such that $t_m = \max(t_k,t_j)$. Then, because $t_k \in Y$ and $t_j \in U$, we know that $t_m \in V$ and that $h \oplus g \in C(\{f_i \mid i \not= m\}) \cup \overline{\D}$, which is exactly what we needed.
Since this is all uniform we therefore see
\[\alpha(Y) \geq_\M \llbracket \phi(x_1,\dots,x_n) \rrbracket_{\langle a_1,\dots,a_n\rangle}.\]

Conversely, take any element
\[(e \conc g) \oplus h \in (\alpha(U) \to_\M \alpha(V)) \oplus \B_{(a_1,\dots,a_n)} = \alpha(U) \to_{[\B_{(a_1,\dots,a_n)},\A]} \alpha(V).\]
We need to compute an element of $\alpha(Y)$. Let $Z$ be the collection of nodes where $a_1,\dots,a_n$ are all in the domain. Then $h$ computes some element $\tilde{h} \in \alpha(Z)$, as follows from the definition of $\B_{(a_1,\dots,a_n)}$ and the fact that we have already proven the claim for conjunctions applied to $\llbracket x_1 = x_1 \wedge \dots \wedge x_n = x_n\rrbracket_{\langle a_1,\dots,a_n\rangle}$. If the second bit of $\tilde{h}$ is $1$, then $\tilde{h}$ computes an element of $\A$ and therefore also computes an element of $\alpha(Y)$. So, we may assume it is $0$. Let $k = \tilde{h}(0)$. First compute if $\phi(a_1,\dots,a_n)$ holds in $\gK$ at the node $t_k$; if so, we know that $\tilde{h} \in \alpha(Y)$ so we are done. Otherwise, there must be a node $t_{\tilde{k}}$ (above $t_k$) such that $t_{\tilde{k}} \in U$ but $t_{\tilde{k}} \not\in V$.

Let $\sigma$ be the least string such that $\Phi(e)\left(g \oplus \left(\tilde{k} \conc 0 \conc \sigma\right)\right)(0){\downarrow}$ and such that $\Phi(e)\left(g \oplus \left(\tilde{k} \conc 0 \conc \sigma\right)\right)(1){\downarrow}$ and let $m = \Phi_e\left(g \oplus \left(\tilde{k} \conc 0 \conc \sigma\right)\right)(0)$ (such a $\sigma$ much exist, since there is some initial segment of $\tilde{k} \conc 0 \conc f_{\tilde{k}+1} \in \alpha(U)$ for which this must halt by choice of $g$ and $e$).
Then we see, by choice of $g$ and $e$ that $t_m \in V$ and that
\begin{align*}
\{g\} \oplus C\left(\left\{f_i \mid i \not= \tilde{k}\right\}\right) &\geq_\M \{g\} \oplus \left(\sigma \conc C\left(\left\{f_i \mid i \not= \tilde{k}\right\}\right) \right)\\
&\geq_\M \left(C(\{f_i \mid i \not= m\}) \cup \overline{\D}\right) \otimes \A.
\end{align*}
In fact, since the value at $1$ has also already been decided by choice of $\sigma$, we even get that either
\[\{g\} \oplus C\left(\left\{f_i \mid i \not= \tilde{k}\right\}\right) \geq_\M \A\]
or
\[\{g\} \oplus C\left(\left\{f_i \mid i \not= \tilde{k}\right\}\right) \geq_\M C(\{f_i \mid i \not= m\}) \cup \overline{\D}.\]
In the first case, we are clearly done. Otherwise, we claim: $g \oplus \tilde{h} \in C(\{f_i \mid i \not= m) \cup \overline{\D}$.
We distinguish several cases:
\begin{itemize}
\item If $\tilde{h} \in \overline{\D}$, then $g \oplus \tilde{h} \geq_T \tilde{h} \in \overline{\D}$ and $\overline{\D}$ is upwards closed.
\item Otherwise, $\tilde{h} \geq_T f_i$ for some $i \not= k$. If $i \not= \tilde{k}$, then we have just seen that $g \oplus \tilde{h}$ computes an element of $C(\{f_i \mid i \not= m\}) \cup \overline{\D}$. Since the latter is upwards closed, we see that $g \oplus \tilde{h} \in C(\{f_i \mid i \not= m\}) \cup \overline{\D}$.
\item If $\tilde{h} \geq_T f_{\tilde{k}}$, then $g \oplus \tilde{h} \geq_T \tilde{h} \in C(\{f_i \mid i \not=m)$: after all, $t_m \in V$ while $t_{\tilde{k}} \not\in V$, so $\tilde{k} \not= m$.
\end{itemize}
Thus, $g \oplus \tilde{h}$ uniformly computes an element of $\alpha(Y)$, which is what we needed to show.

Now, let us consider the quantifiers. So, let $\phi(x_1,\dots,x_n) = \forall y \psi(x_1,\dots,x_n,y)$. For every $b \in \omega$, let $U_b$ be the set of nodes where $\psi(a_1,\dots,a_n,b)$ holds in $\gK$, and likewise let $Y$ be the set of nodes where $\phi(a_1,\dots,a_n)$ holds. We need to show that
\[\llbracket \phi(x_1,\dots,x_n) \rrbracket_{\langle a_1,\dots,a_n\rangle} \equiv_\M \alpha(Y).\]
By definition of the interpretation of the universal quantifier and the induction hypothesis, we know that
\begin{align*}
\llbracket \phi(x_1,\dots,x_n) \rrbracket_{\langle a_1,\dots,a_n\rangle} &\equiv_\M \left(\bigoplus_{b \in \omega} \B_{(a_1,\dots,a_n,b)} \to_\M \alpha(U_b)\right) \oplus \B_{a_1} \oplus \dots \oplus \B_{a_n}\\
&= \bigoplus_{b \in \omega} \B_{(a_1,\dots,a_n,b)} \to_{[\B_{(a_1,\dots,a_n)},\A]_\M} \alpha(U_b).\\
\intertext{Let $Z_b$ be the set of nodes where $a_1,\dots,a_n$ and $b$ are in the domain, and let $Z$ be the set of nodes where $a_1,\dots,a_n$ are in the domain. Then we get in the same way as above:}
\llbracket \phi(x_1,\dots,x_n) \rrbracket_{\langle a_1,\dots,a_n\rangle} &\equiv_\M \bigoplus_{b \in \omega} \alpha(Z_b) \to_{[\B_{(a_1,\dots,a_n)},\A]_\M} \alpha(U_b).\\
\intertext{Finally, let us introduce new predicates $R_b(x_1,\dots,x_n)$, which are defined to hold in $\gK$ if $\phi(x_1,\dots,x_n,b)$ holds in $\gK$, and let us introduce new nullary predicates $S_b$ which are defined to hold when all of $a_1,\dots,a_n$ and $b$ are in the domain. Then, applying the fact that we have already proven the claim for implications to $\llbracket S_b \to R_b\rrbracket_{\langle a_1,\dots,a_n\rangle}$, we get}
\llbracket \phi(x_1,\dots,x_n) \rrbracket_{\langle a_1,\dots,a_n\rangle} &\equiv_\M \bigoplus_{b \in \omega} \alpha((Z_b \to U_b) \cap Z).
\end{align*}
We now claim that this is equivalent to $\alpha(Y)$. We have $Y \subseteq (Z_b \to U_b) \cap Z$ by the definition of truth in Kripke frames, which suffices to prove that
\[\bigoplus_{b \in \omega} \alpha((Z_b \to U_b) \cap Z) \leq_\M \alpha(Y).\]

Conversely, let
\[\bigoplus_{b \in \omega} g_b \in \bigoplus_{b \in \omega} \alpha((Z_b \to U_b) \cap Z).\]
We show how to compute an element of $\alpha(Y)$ from this.
If the second bit of $g_0$ is $1$, then $h$ computes an element of $\A$; thus, assume it is $0$.
Let $m_0 = g_0(0)$. First compute if $\phi(a_1,\dots,a_n)$ holds in $\gK$ at the node $\gamma(t_{m_0})$; if so, we know that $g_0 \in \alpha(Y)$ so we are done. Therefore, we may assume this is not the case. So, we can compute a $b_1 \in \omega$ such that $t_{m_0} \not\in Z_{b_1} \to U_{b_1}$ by the definition of truth in Kripke frames. Now consider $g_{b_1}$. If the second bit of $g_{b_1}$ is $1$, then $g_{b_1}$ computes an element of $\A$ so we are done. Otherwise, let $m_1 = g_{b_1}(0)$. Then $t_{m_1} \in Z_{b_1} \to U_{b_1}$ and $g_{b_1} \in C(f_i \mid i \not= m_1) \cup \overline{\D}$. Then $m_1 \not\leq m_0$ because $t_{m_0} \not\in Z_{b_1} \to U_{b_1}$. If $m_1$ is incomparable with $m_0$, then $m_0 \conc m_1 \conc (g_{b_1} \oplus h) \in \A$ so we are done. Thus, the only remaining case is when $m_1 > m_0$.

Iterating this argument, if it does not terminate after finitely many steps, we obtain a sequence $m_0 < m_1 < m_2 < \dots$. However, we assumed that our Kripke frame does not contain any infinite ascending chains, so the algorithm has to terminate after finitely many steps. Thus, 
\[\bigoplus_{b \in \omega} \alpha((Z_b \to U_b) \cap Z) \geq_\M \alpha(Y).\]
We note that this is the only place in the proof where we use the assumption about infinite ascending chains.

Finally, we consider the existential quantifier. To this end, let $\phi(x_1,\dots,x_n) = \exists y \psi(x_1,\dots,x_n,y)$. Let $U_b$ and $Z$ be as for the universal quantifier. Then the induction hypothesis tells us that
\[\llbracket \phi(x_1,\dots,x_n) \rrbracket_{\langle a_1,\dots,a_n\rangle} \equiv_\M \bigcup\{b \conc \alpha(Y_b) \mid b \in \omega\}.\]
First, since $Y_b \subseteq Z$, we certainly have that $\alpha(Z) \leq_\M \bigcup\{b \conc \alpha(Y_b) \mid b \in \omega\}$. Conversely, let $j \conc f \in \alpha(Z)$. Then $f \in \left(C(f_i \mid i \not= j) \cup \overline{\D}\right) \otimes \A$ and $t_j \in Z$. Thus, there is some $b \in \omega$ such that $\psi(a_1,\dots,a_n,b)$ holds, and therefore by induction hypothesis $j \conc f \in \alpha(Y_b)$. Furthermore, since $\gK$ is decidable we can compute such a $b$. Thus, $\alpha(Z) \geq_\M \bigcup\{b \conc \alpha(Y_b) \mid b \in \omega\}$, which completes the proof of the claim.

\bigskip
Thus, by the claim we have that, for any sentence $\phi$, that $\llbracket \phi \rrbracket = \alpha(Y)$, where $Y$ is the set of nodes where $\phi$ holds in the Kripke model $\gK$. Furthermore, $\alpha$ is injective so $\alpha(Y) = \B_{-1}$ if and only if $Y = T$. So, $\phi$ holds in $\gM$ if and only if $Y = T$ if and only if $\phi$ holds in $\gK$, which is what we needed to show.

For the second part of the theorem, note that we only used the assumption about infinite ascending chains in the part of the proof dealing with the universal quantifier.
\end{proof}

\begin{lem}\label{lem-inf-meet-upw}
Let $\C \subseteq \omega^\omega$ be non-empty and upwards closed under Turing reducibility, let $\E_i \subseteq \omega^\omega$ and let $\bigcup \{i \conc \E_i\} \leq_\M \C$. Then there is an $i \in \omega$ such that $\E_i \leq_\M \C$.
\end{lem}
\begin{proof}
Let $\Phi_e(\C) \subseteq \bigcup \{i \conc \E_i\}$. Let $\sigma$ be the least string such that $\Phi_e(\sigma)(0){\downarrow}$. Such a string must exist, because $\C$ is non-empty. Let $i = \Phi_e(\sigma)(0)$. Then:
\[\C \geq_\M \sigma \conc \C \geq_\M \E_i,\]
as desired.
\end{proof}

Our proof relativises if our language does not contain function symbols, which gives us the following result.

\begin{thm}\label{thm-dec-fram-rel}
Let $\gK$ be a Kripke model for a language without function symbols which is based on a Kripke frame without infinite ascending chains. Then there is an interval $[\hiB,\A]_{\P_\M}$ and a structure $\gM$ in $[\hiB,\A]_{\P_\M}$ such that the theory of $\gM$ is exactly the theory of $\gK$.

Furthermore, if we allow infinite ascending chains, then this still holds for the fragments of the theories without universal quantifiers.
\end{thm}
\begin{proof}
Let $h$ be such that $\gK$ is $h$-decidable. We relativise the construction in the proof of Theorem \ref{thm-dec-fram} to $h$. We let all definitions be as in that proof, except where mentioned otherwise.
This time we let $f_i$ be an antichain over $h$, i.e.\ for all $i \not= j$ we have $f_i \oplus h \not\geq_T f_j$. We change the definition of $\D$ into $\{g \mid \exists i (g \leq_T f_i \oplus h)\}$
We let
\begin{align*}
\A = \{\left(k_1 \conc k_2 \conc \left(C(\{f_i \mid i \not\in \{k_1,k_2\}\}) \cup \overline{\D}\right)\right) \oplus h \mid \text{$t_{k_1}$ and $t_{k_2}$ are incomparable}\}
\end{align*}
if $T$ is not a chain, and let $\A = \overline{\D} \oplus h$ otherwise.
We let $\beta(Y) = \alpha(Y) \oplus \{h\}$ for all $Y \in \U$. Then $\beta$ is still injective. Indeed, let us assume $\beta(Y) \leq_\M \beta(Z)$; we will show that $Y \supseteq Z$. By applying Lemma \ref{lem-inf-meet-upw-rel} below we see that for every $j$ with $t_j \in Z$ there exists a $k$ with $t_k \in Y$ such that either $\left(C(\{f_i \mid i \not= k\}) \cup \overline{\D}\right) \oplus \{h\}  \leq_\M \left(C(\{f_i \mid i \not= j\}) \cup \overline{\D}\right) \oplus \{h\}$ or $\A \leq_\M \left(C(\{f_i \mid i \not= j\}) \cup \overline{\D}\right) \oplus \{h\}$. If the first holds, let us assume that $k \not= j$; we will derive a contradiction from this. Then $f_k \in C(\{f_i \mid i \not= j\}) \cup \overline{\D}$ and therefore $f_k \oplus h \in C(\{f_i \mid i \not= k\}) \cup \overline{\D}$ since this set is upwards closed. However, we know that the $f_i$ form an antichain over $h$ in the Turing degrees, which is a contradiction. So, $k = j$ and therefore $t_j \in Y$.

In the second case, we have that
\[C(\{f_i \mid i \not\in \{k_1,k_2\}) \cup \overline{\D} \leq_\M \left(C(\{f_i \mid i \not= j\}) \cup \overline{\D}\right) \oplus \{h\}\]
for some $k_1,k_2 \in \omega$ for which $t_{k_1}$ and $t_{k_2}$ are incomparable. Without loss of generality, we may assume that $k_1 \not= j$. Then, in the same way as above, we see that $f_{k_1} \oplus h \in C(\{f_i \mid i \not\in \{k_1,k_2\}) \cup \overline{\D}$ which is again a contradiction.

We let $\B_{-1} = \beta(T)$ and we let $\B_i = \beta(Z_i)$, where $Z_i$ is the set of nodes where $i$ is in the domain of $\gK$.
We claim: for every formula $\phi(x_1,\dots,x_n)$ and every sequence $a_1,\dots,a_n$,
\[\llbracket \phi(x_1,\dots,x_n) \rrbracket_{\langle a_1,\dots,a_n\rangle} \equiv_\M \beta(Y),\]
where $Y$ is exactly the set of nodes where $a_1,\dots,a_n$ are all in the domain and $\phi(a_1,\dots,a_n)$ holds in the Kripke model $\gK$. The proof is the same as before, except that this time we use that all mass problems we deal with are above $\B_{-1} = \alpha(T) \oplus \{h\}$ and hence uniformly compute $h$. Thus, we can still decide all the properties about $\gK$ which we need during the proof.
\end{proof}

\begin{lem}\label{lem-inf-meet-upw-rel}
Let $\C \subseteq \omega^\omega$ be non-empty and upwards closed under Turing reducibility, let $\E_i \subseteq \omega^\omega$, let $h \in \omega^\omega$ and let $\bigcup \{i \conc \E_i\} \leq_\M \C \oplus \{h\}$. Then there is an $i \in \omega$ such that $\E_i \leq_\M \C$.
\end{lem}
\begin{proof}
Let $\Phi_e(\C) \subseteq \bigcup \{i \conc \E_i\}$. Let $\sigma$ be the least string such that $\Phi_e(\sigma \oplus h)(0){\downarrow}$. Such a string must exist, because $\C$ is non-empty. Let $i = \Phi_e(\sigma \oplus h)(0)$. Then:
\[\C \oplus h \geq_\M (\sigma \conc \C) \oplus h \geq_\M \E_i,\]
as desired.
\end{proof}

We will now use Theorem \ref{thm-dec-fram-rel} to show that we can refute the formulas discussed in section \ref{sec-theory}.

\begin{prop}\label{prop-refute-1}
There is an interval $[\hiB,\A]_{\P_\M}$ and a structure $\gM$ in $[\hiB,\A]_{\P_\M}$ such that $\gM$ refutes the formula
\[\forall x,y,z (x = y \vee x = z \vee y = z) \wedge \forall z(S(z) \vee R) \to \forall z(S(z)) \vee R\]
from Proposition \ref{prop-cd}.
\end{prop}
\begin{proof}
As shown in the proof of Proposition \ref{prop-cd} there is a finite Kripke frame refuting the formula. Now apply Theorem \ref{thm-dec-fram-rel}.
\end{proof}

\begin{prop}\label{prop-refute-2}
There is an interval $[\hiB,\A]_{\P_\M}$ and a structure $\gM$ in $[\hiB,\A]_{\P_\M}$ such that $\gM$ refutes the formula 
\[(\forall x (S(x) \vee \neg S(x))  \wedge \neg \forall x (\neg S(x))) \to \exists x (\neg\neg S(x)).\]
from Proposition \ref{prop-cd-2}.
\end{prop}
\begin{proof}
In the proof of Proposition \ref{prop-cd-2} we showed that there is a finite Kripke frame refuting the given formula. So, the claim follows from Theorem \ref{thm-dec-fram-rel}.
\end{proof}

Thus, moving to the more general intervals $[\hiB,\A]_{\P_\M}$ did allow us to refute more formulas. Let us next note that Theorem \ref{thm-not-iqc} really depends on the fact that we chose the language of arithmetic to contain function symbols.

\begin{prop}
Let $\Sigma$ be the language of arithmetic, but formulated with relations instead of with function symbols. Let $T$ be derivable in $\mathrm{PA}$ and let $\chi$ be a $\Pi^0_1$-sentence or $\Sigma^0_1$-sentence which is not derivable in $\mathrm{PA}$.
Then there is an interval $[\hiB,\A]_{\P_\M}$ and a structure $\gM$ in $[\hiB,\A]_{\P_{\M}}$ refuting $\bigwedge T \to \chi$.
\end{prop}
\begin{proof}
Let $\gK$ be a classical model refuting $\bigwedge T \to \chi$, which can be seen as a Kripke model on a frame consisting of one point. Now apply Theorem \ref{thm-dec-fram-rel}.
\end{proof}

Finally, let us consider the schema $\forall x \neg\neg\phi(x) \to \neg\neg\forall x \phi(x)$, called \emph{Double Negation Shift (DNS)}. It is known that this schema characterises exactly the Kripke frames for which every node is below a maximal node (see Gabbay \cite{gabbay-1981}), so in particular it holds in every Kripke frame without infinite chains. We will show that we can refute it in an interval of the hyperdoctrine of mass problems, even though Theorem \ref{thm-dec-fram-rel} does not apply.

\begin{prop}
Let $\Sigma$ be the language containing one unary relation $R$. There is an interval $[\hiB,\A]_{\P_\M}$ and a structure $\gM$ in $[\hiB,\A]_{\P_\M}$ such that $\gM$ refutes $\forall x \neg\neg R(x) \to \neg\neg\forall x R(x)$.
\end{prop}
\begin{proof}
We let $\gK$ be the Kripke model based on the Kripke frame $(\omega,<)$, where $n$ is in the domain at $m$ if and only if $m \geq n$, and $R(n)$ holds at $m$ if and only if $m > n$.
Let everything be as in the proof of Theorem \ref{thm-dec-fram}, except we change the definition of $\A$ into:
\[\bigcup\left\{\left(C\left(\{f_i \mid i \not\in X\right\}) \cup \overline{\D}\right) \oplus X \mid X \in 2^\omega \text{ is infinite}\right\},\]
where by $X$ being infinite we mean that the subset $X \subseteq \omega$ represented by $X$ is infinite.
We claim: $\alpha$ is still injective under this modified definition of $\A$.
Indeed, assume that
\[\A \leq_\M C(\{f_i \mid i \not= j\}) \cup \overline{\D},\]
say through $\Phi_e$; we need to show that this still yields a contradiction.
Let $\sigma$ be the least string such that the right half of $\Phi_e(\sigma)$ has a $1$ at a position different from $j$, say at position $k$; such a $\sigma$ must exist since $\Phi_e(f_{j+1}) \in \A$. Then $\Phi_e(\sigma \conc f_{k}) \in C(\{f_i \mid i \not= k\}) \cup \overline{\D}$, which is a contradiction.

All the other parts of the proof of Theorem \ref{thm-dec-fram} now go through as long as we look at formulas not containing existential quantifiers. Since $\forall x \neg\neg R(x)$ is intuitionistically equivalent to $\neg \exists x \neg R(x)$, we therefore see that
\[\llbracket \forall x \neg\neg R(x) \rrbracket \equiv_\M \B_{-1}.\]
We claim: $\llbracket \neg\forall x (R(x)) \rrbracket \equiv_\M \B_{-1}$, which is enough to prove the proposition. Note that $\llbracket \forall x (R(x)) \rrbracket \equiv_\M \B_{-1} \oplus \bigoplus_{m \in \omega} (\B_m \to_\M \B_{m+1})$. By introducing new predicates $S_m$ which hold if and only if $m$ is in the domain and looking at $\llbracket S_m \to S_{m+1}\rrbracket$, we therefore get that 
$\llbracket \forall x (R(x)) \rrbracket \equiv_\M \bigoplus_{m \in \omega} \B_{m+1}$.

We claim that from every element $g \in \bigoplus_{m \in \omega} \B_{m+1}$ we can uniformly compute an element of $\A$. In fact, we show how to uniformly compute from $g$ a sequence $k_0 < k_1 < \dots$ such that $g  \in C(\{f_i \mid i \not= k_j\}) \cup \overline{\D}$ for every $j \in \omega$; then if we let $X = \{k_j \mid j \in \omega\}$ we have $g \oplus X \in \left(C(\{f_i \mid i \not\in X) \cup \overline{\D}\right) \oplus X \subseteq \A$.
For ease of notation let $k_{-1} = 0$. We show how to compute $k_{i+1}$ if $k_i$ is given.
There are two possibilities:
\begin{itemize}
\item The second bit of $g^{[k_i]}$ is $0$: take $k_{i+1}$ to be the first bit of $g^{[k_i]}$; then $k_{i+1} > k_i$ by the definition of $\B_{k_i+1}$.
\item The second bit of $g^{[k_i]}$ is $1$: then $g^{[k_i]}$ computes an element of $\A$ and therefore computes infinitely many $j$ such that 
$g^{[k_i]} \in C(\{f_i \mid i \not= j\}) \cup \overline{\D}$, so take $k_{i+1}$ to be such a $j$ which is greater than $k_i$.\qedhere
\end{itemize}
\end{proof}

We do not know how to combine the proof of the last Proposition with the proofs of Theorems \ref{thm-dec-fram} and \ref{thm-dec-fram-rel}, because it makes essential use of the fact that the formula is refuted in a model on a frame which is a chain, and of the fact that the subformulas containing universal quantifiers hold either everywhere or nowhere in this model. Table \ref{table-decidable} below summarises the positive results we know; however, this characterisation is not complete.

\begin{question}
For which theories $T$ is there an interval $[\hiB,\A]_{\P_\M}$ and a structure $\gM$ in $[\hiB,\A]_{\P_\M}$ such that the theory of $\gM$ is exactly $T$?
\end{question}

\begin{table}[H]
\begin{tabular}{cccc}
\toprule
Theorem & Language & Fragment & Kripke frame condition\\
\midrule
\ref{thm-dec-fram} & Arbitrary & Full & Decidable and no infinite chains\\
\ref{thm-dec-fram} & Arbitrary & Existential & Decidable\\
\ref{thm-dec-fram-rel} & No functions & Full & No infinite chains\\
\ref{thm-dec-fram-rel} & No functions & Existential & None\\
\bottomrule
\end{tabular}
\caption{Fragments of theories which have a structure in some interval $[\hiB,\A]_{\P_\M}$, given their satisfiability by a certain kind of Kripke frame.}
\label{table-decidable}
\end{table}

\providecommand{\bysame}{\leavevmode\hbox to3em{\hrulefill}\thinspace}
\providecommand{\MR}{\relax\ifhmode\unskip\space\fi MR }
\providecommand{\MRhref}[2]{%
  \href{http://www.ams.org/mathscinet-getitem?mr=#1}{#2}
}
\providecommand{\href}[2]{#2}

\end{document}